\documentclass[10pt,letter]{article}

\usepackage[total={6in, 9in}]{geometry}

\usepackage{amsmath}
\usepackage{amssymb}
\usepackage{amsthm}
\usepackage{booktabs}
\usepackage {verbatim}
\usepackage{epsfig,psfrag}
\usepackage{algorithm}
\usepackage{algorithmic}
\usepackage{url}
\usepackage{float}
\usepackage{mathtools}
\usepackage{graphicx}
\usepackage[caption=false]{subfig}
\usepackage{enumerate}
\usepackage{textgreek}
\usepackage{bbm}
\usepackage{comment}
\usepackage{multirow}

\allowdisplaybreaks

\newcommand{\beq}{\begin{equation}}
\newcommand{\eeq}{\end{equation}}
\newcommand{\bal}{\begin{aligned}}
\newcommand{\eal}{\end{aligned}}

\newcommand{\setC}{{\cal C}}

\newcommand{\setU}{{\cal U}}
\newcommand{\setK}{{\cal K}}
\newcommand{\setV}{{\cal V}}

\newcommand{\tA}{\widetilde{A}}
\newcommand{\tComplState}{\widetilde{\ComplState}}
\newcommand{\tStateCompl}{\widetilde{\StateCompl}}
\newcommand{\tComplCompl}{\widetilde{\ComplCompl}}
\newcommand{\tP}{\widetilde{P}}

\newcommand{\hatQ}{\widehat{Q}}

\newcommand{\Sphere}{\mathcal{S}}
\newcommand{\R}{\mathbb{R}}
\newcommand{\Sy}{\mathbb{S}}
\newcommand{\Z}{\mathbb{Z}}
\newcommand{\bfP}{\mathbf{P}}
\newcommand{\bfM}{\mathbf{M}}

\newcommand{\dlcs}{DLCS}
\newcommand{\ComplState}{C}

\newcommand{\StateCompl}{D}

\newcommand{\ComplCompl}{F}

\newcommand{\lcp}[2]{\text{LCP}(#1,#2)}
\newcommand{\sol}[2]{\text{SOL}(#1,#2)}
\newcommand{\gr}{\text{Gr}}

\newcommand{\tx}{\widetilde{x}}

\newcommand{\dirlambda}{\lambda'}
\newcommand{\hatlambda}{\widehat{\lambda}}

\newcommand{\mysub}[2]{{\left[#1\right]}_{#2}}

\newcommand{\conv}{\textnormal{conv}}

\newcommand{\STAB}[1]{\begin{tabular}{@{}c@{}}#1\end{tabular}}
\newcommand{\tstep}{h}
\newcommand{\tAthetah}{\tA_{\theta,h}}

\newcommand{\xtraj}{x^{\mathrm{lcs}}}
\newcommand{\lamtraj}{\lambda^{\mathrm{lcs}}}

\newcommand{\exclude}[1]{}

\newtheorem{lemma}{Lemma}
\newtheorem{theorem}{Theorem}
\newtheorem{definition}{Definition}

\newtheorem{proposition}{Proposition}

\title{Stability Analysis of Discrete-Time Linear Complementarity Systems}

\author{Arvind U. Raghunathan\thanks{Mitsubishi Electric Research Laboratories
        (\texttt{raghunathan@merl.com})}.        
        \and Jeffrey T. Linderoth\thanks{Department of Industrial and Systems Engineering,
        University of Wisconsin-Madison
        (\texttt{linderoth@wisc.edu}).}}

\usepackage{amsopn}

\begin{document}

\maketitle

\begin{abstract}
A Discrete-Time Linear Complementarity System (\dlcs) is a dynamical system in discrete time whose state evolution is governed by linear dynamics in states and algebraic variables that solve a Linear Complementarity Problem (LCP). The \dlcs\, is the hybrid dynamical system that is the discrete-time counterpart of the well-known Linear Complementarity System (LCS).  We derive sufficient conditions for Lyapunov stability of a \dlcs\, when using a quadratic Lyapunov function that depends only on the state variables and a quadratic Lyapunov function that depends both on the state and the algebraic variables.  The sufficient conditions require checking the feasibility of a copositive program over nonconvex cones.  Our results only assume that the LCP is solvable and do not require the solutions to be unique.  We devise a novel, exact cutting plane algorithm for the verification of stability and the computation of the Lyapunov functions.  To the best of our knowledge, our algorithm is the first exact approach for stability verification of \dlcs. A number of numerical examples are presented to illustrate the approach.  Though our main object of study in this paper is the \dlcs, the proposed algorithm can be readily applied to the stability verification of LCS.  In this context, we show the equivalence between the stability of a LCS and the \dlcs\, resulting from a time-stepping procedure applied to the LCS for all sufficiently small time steps.
\end{abstract}



\section{Introduction}

A Discrete-Time Linear Complementarity System (\dlcs) is linear dynamical system in a discrete time where the state evolution is governed by \emph{complementarity} constraints.  Mathematically, a \dlcs\, is written as
\begin{subequations}
	\begin{align}
		x_{k+1} =&\; Ax_k + \ComplState\lambda_{k+1} \label{dlcs:dyn}\\
		0 \leq \lambda_{k+1} \perp&\; \StateCompl x_{k} + \ComplCompl 
		\lambda_{k+1} \geq 0, \label{dlcs:compl}
	\end{align}\label{dlcs}
\end{subequations}
where $x_k \in \R^{n_x}$ is the state of the system at the time-step $k$ and \emph{complementarity variables} $\lambda_k \in \R^{n_c}$ are the algebraic variables satisfying complementarity constraints in~\eqref{dlcs:compl}. The dimensions of the matrices in~\eqref{dlcs} are $A \in \R^{n_x \times n_x}$, $\ComplState \in \R^{n_x \times n_c}$, $\StateCompl \in \R^{n_c \times n_x}$, and $\ComplCompl \in \R^{n_c \times n_c}$.  The complementarity constraints~\eqref{dlcs:compl} distinguish the \dlcs\, from a standard linear dynamical system. A \dlcs\, can be naturally derived by time sampling a Linear Complementarity System (LCS)~\cite{HeemelsThesis2000, HeemelsSchumacher2000, CamlibelPangShen2006}.

A LCS is a continuous-time linear dynamical systems governed by \emph{complementarity} constraints and can be written mathematically as follows:
\begin{subequations}
	\begin{align}
		\frac{dx}{dt} =&\; \tA x(t) + \tComplState \lambda(t) \label{lcs:dyn}\\
		0 \leq \lambda(t) \perp&\; \tStateCompl x(t) + \tComplCompl \lambda(t) \geq 0, \label{lcs:compl}
	\end{align}\label{lcs}
\end{subequations}
where the matrices $\tA, \tComplState, \tStateCompl, \tComplCompl$ have conformal dimensions.  An appropriate time-stepping scheme such as \emph{explicit} or \emph{implicit} Euler~\cite{PangStewart2008} can be shown to yield the \dlcs~\eqref{dlcs}. 
A number of authors have studied time-stepping methods for \eqref{lcs} and established sufficient conditions for convergence of numerical methods \cite{CamlibelHeemels2003,PangStewart2008,HanTiwariCamlibel2009}.
For example, Pang and Stewart~\cite[\S 3]{PangStewart2008} propose simulating a LCS using time-stepping formulations of the form 
\begin{subequations}
	\begin{align}
		x_{k+1} =&\; x_{k} + \tstep (\tA x_{k+\theta} + \tComplState \lambda_{k+1} )\label{tslcs:dyn}\\
		0 \leq \lambda_{k+1} \perp&\; \tStateCompl x_{k+1} + \tComplCompl \lambda_{k+1} \geq 0, \label{tslcs:compl}
	\end{align}\label{tslcs}
\end{subequations}
where $\tstep$ is the time-step, $\theta \in [0,1]$ and $x_{k+\theta} = (1-\theta)x_k + \theta x_{k+1}$.  The dynamics of the time-stepping formulations can be transformed to a \dlcs\, by rearranging~\eqref{tslcs:dyn} as
\begin{subequations}
	\begin{align}
		x_{k+1} =&\; (I_{n_x} + \tstep \tAthetah^{-1}\tA) x_{k} + \tstep \tAthetah^{-1} \tComplState \lambda_{k+1} \label{tslcs1:dyn}\\
		0 \leq \lambda_{k+1} \perp&\; \tStateCompl x_{k+1} + \tComplCompl \lambda_{k+1} \geq 0 \label{tslcs1:compl}
	\end{align}\label{tslcs1}
\end{subequations}
where $\tAthetah = (I_{n_x} - \theta \tstep \tA)$.  Specifically, substituting for $x_{k+1}$ in~\eqref{tslcs1:compl} using~\eqref{tslcs1:dyn}, we obtain the \dlcs\,~\eqref{dlcs}. The use of $\theta = 0,\frac{1}{2},1$ results in the explicit Euler, trapezoidal, or implicit Euler time-stepping schemes, respectively.  Convergence of the solutions of the time-stepping formulations~\eqref{tslcs} as $\tstep \rightarrow 0$ to a solution of the LCS~\eqref{lcs} are provided in~\cite{PangStewart2008}.

In dynamical systems, the discrete-time sampled variant plays an important role in implementing control algorithms in practice.  For example, discrete-time control algorithms, such as model predictive control (MPC)~\cite{Mayne2000}, rely on the time-stepping formulations of continuous-time linear dynamical systems.  In an analogous manner, the \dlcs\, naturally arises in applications that can be modeled with LCS.  The modeling and study of dynamical systems with complementarity conditions has been steadily increasing, with application in  electrical circuit simulation, robotics, nonsmooth mechanics, economics, bioengineering  \cite{AcaryBrogliatoBook2008, Brogliato2003, CamlibelThesis2001,  CamlibelHeemels2003,  NagurneyZhang1996, PangShen2007, RaghunathanDVI2006, Schumacher2004, VanSchaftSchumacher1998, VanSchaftSchumacher2000, VascaIannelli2009}.  Another important recent application of interest is robotics, where complementarity is required to model the friction forces that arise in contact-based manipulation tasks~\cite{KennethPilwon2017, ManchesterKuindersma2017, MordatchTodorov2012, PatelShield2019, posa_thesis, Posa14}.  

Stability is a fundamental issue in dynamical systems, and stability verification by computing stability certificates is a key to guaranteeing performance and safety in real-world systems.  Lyapunov stability is a widely used concept for analyzing the stability of dynamical systems~\cite{khalil2002nonlinear}, and its extension to hybrid and switched systems has been considered by several authors~\cite{Branicky1998, DecarloBranicky2000, Hespanha2004, Hespanha2005, JohanssonRantzer1998, Liberzon2003, Shorten2007}.  The  papers extending stability to hybrid systems assume that a Lyapunov-like function exists for each mode’s vector field and for the entire state space. In many hybrid and switched systems, however, each mode is active only over a subset of the state space, especially for those systems whose switchings are triggered by state evolution, such as in the LCS and \dlcs. Hence, the above results are rather restrictive, even for linear switched systems. To rectify this, {\c C}amlibel and Schumacher~\cite{CamlibelSchumacher2004} proposed copositive Lyapunov functions for conewise linear systems in
which the feasible region of each mode is a polyhedral cone. Bundfuss and D\"{u}r~\cite{BundfussDur2009} present necessary and sufficient conditions for the existence of such Lyapunov functions when the cone is polyhedral.   
{\c C}amlibel, Pang and Shen~\cite{CamlibelPangShen2006} derived sufficient conditions for the stability of LCS~\eqref{lcs}.  Extending standard approaches that employ piecewise-quadratic Lyapunov functions depending only on state $x(t)$, the authors in~\cite{CamlibelPangShen2006} proposed an extended quadratic Lyapunov function that depends on both $x(t)$ and $\lambda(t)$.  However, the authors did not provide an algorithm for computing such a Lyapunov function.  Recently, Aydinoglu, Preciado, and Posa~\cite{Aydinoglu2020} extended the sufficient conditions of~\cite{CamlibelPangShen2006}  under the assumption of an existence of a feedback controller that depends on both $x(t)$ and $\lambda(t)$.  The authors also proposed an algorithm based on bilinear matrix inequalities for jointly determining the feedback controller and the Lyapunov function.  However, we are not aware of similar results for the \dlcs.

In this paper, we focus on the stability analysis of the \dlcs\, and the development of algorithms for computing a Lyapunov function certifying stability.  The paper makes three main contributions.

First, in \S\ref{sec:stability} we derive sufficient conditions for the stability of a \dlcs\, using two quadratic Lyapunov functions.  The first Lyapunov function is a quadratic function that depends only on the states of the \dlcs~\eqref{dlcs}.  Such a Lyapunov function is referred to as a \emph{Common Quadratic Lyapunov Function} (CQLF)~\cite{Liberzon2003, Shorten2007}.  The second Lyapunov function is a quadratic function that depends on both the states and complementarity variables of the \dlcs~\eqref{dlcs}.  This is the discrete-time analogue of the \emph{Extended Quadratic Lyapunov Function} (EQLF) introduced in~\cite{CamlibelPangShen2006}.  The EQLF is a piecewise-quadratic function of the state variables.  The stability conditions are provided under the assumption that $\ComplCompl$ is a Q-matrix and a R$_0$-matrix~\cite{CottlePangStone}.  The Q-matrix is the largest class of matrices for which the complementarity constraints in~\eqref{dlcs:compl} have a solution for all $x_k$.  Our assumptions on the matrix $F$ are weaker than those used in~\cite{Aydinoglu2020,CamlibelPangShen2006}.

Second, in \S\ref{sec:algorithm} we present exact algorithms for computing the CQLF and EQLF for a \dlcs. The computation of Lyapunov functions requires finding a feasible solution for a copositive program that is defined over a nonconvex cone---in particular, the union of convex cones.  It is well-known that checking whether a given matrix is copositive is a co-NP-complete problem~\cite{MurtyKabadi1987}.  We pose the problem of computing the Lyapunov functions as semi-infinite programming problem and devise a cutting-plane algorithm.  The cutting plane algorithm alternates a linear programming master problem and a separation problem formulated as a Mixed Integer Quadratically Constrained Program (MIQCP).  Numerical experiments demonstrating the utility of the algorithm are presented in \S\ref{sec:numerics}. 

Third, in \S\ref{sec:lcsdlcsstability} we consider the case where the \dlcs\, is obtained by a time-stepping scheme applied to a LCS.  An important question as yet unanswered in the literature is whether stability of the LCS is equivalent to stability of the \dlcs\, when the time-step $\tstep$ is sufficiently small.  We show equivalence for both the case of CQLF and EQLF under the assumption that the matrix $\tComplCompl$ in the LCS~\eqref{lcs} is a P-matrix.

\emph{Notation.} The set of reals, nonnegative reals, integers and nonnegative integers is denoted by $\R$, $\R_+$, $\Z$ and $\Z_+$ respectively.  The 1-norm sphere is denoted as $\Sphere^n_1 = \{ w \in \R^n \,|\, \|w\|_1 = 1\}$. For a vector $v \in \R^n$ $\mysub{v}{i}$ denotes the $i$-th component of the vector.  
For $v \in \R^n$ the notation $\mysub{v}{\alpha}$ for $\alpha \subset \{1,\ldots,n\}$ denotes the subvector obtained by removing the elements $\mysub{v}{i}$ for $i \notin \alpha$. 
The notation $(x,y)$ of two column vector $x,y$ is also used to represent a vertical stacking of the vectors 
$\begin{psmallmatrix} x \\ y\end{psmallmatrix}$.  
For a matrix $M \in \R^{n \times m}$ $\mysub{M}{ij}$ denotes the $(i,j)$-th entry in the matrix. For a matrix $M$ and index sets $\alpha \subseteq \{1,\ldots,n\}$, $\beta \subseteq \{1,\ldots,m\}$  the notation $\mysub{M}{\alpha\beta}$ denotes the submatrix of $M$ formed by removing rows not in $\alpha$ and removing columns not in $\beta$. The notation $\mysub{M}{\alpha\bullet}$ and $\mysub{M}{\bullet\alpha}$ refers to the submatrix of $M$ obtained by retaining only the rows and columns in $\alpha$ respectively.  The notation $I_n$ represents the $n \times n$ identity matrix. The set $\Sy^n$ denotes the $n \times n$ real, symmetric matrices, and $\Sy^n_+, \Sy^n_{++}$ denote the set of symmetric positive semidefinite and positive definite matrices. A symmetric positive (semi)definite matrix $N \in \Sy^n$ is denoted as $N \succ (\succcurlyeq) 0$.

\section{Background}\label{sec:background}

In this section, we present relevant results and notation that are used throughout the paper.

\subsection{Linear Complementarity Problem (LCP)}
The Linear Complementarity Problem (LCP), denoted as \lcp{$q$}{$M$}, is to find a solution $\lambda \in \R^{n_c}$ of
\begin{equation}
	0 \leq \lambda \perp M\lambda + q \geq 0,	\label{lcp}
\end{equation}
where $q \in \R^{n_c}, M \in \R^{n_c \times n_c}$.   The solution set of the LCP~\eqref{lcp} is denoted as \sol{$q$}{$M$} $= \{ \lambda \,|\, \lambda \text{ satisfies } \eqref{lcp}\}$.  If $q = Nx$ for some matrix $N$, the graph of $\sol{Nx}{M}$ is denoted as $\gr\,\sol{Nx}{M}$ $= \{(x,\lambda) \,|\, \lambda \in \sol{Nx}{M} \}$.  The monograph~\cite{CottlePangStone} provide an in-depth treatment describing conditions for the existence and uniqueness of solutions to (LCP).  We provide a brief summary of definitions and results that will be useful for subsequent developments. 

Given a solution $\lambda \in$ \sol{$q$}{$M$}, we define index sets $\alpha(\lambda)$, $\beta(\lambda)$, $\gamma(\lambda)$ partitioning $\{1,\ldots,n_c\}$ as
\begin{subequations}
	\begin{align}
		\alpha(\lambda) &:=\; \{ i \,|\, \mysub{\lambda}{i} > 0 =  \mysub{M\lambda + q}{i} \} \label{lcpind:alpha} \\
		\beta(\lambda) &:=\; \{ i \,|\, \mysub{\lambda}{i} = 0 =  \mysub{M\lambda + q}{i} \} \label{lcpind:beta} \\
		\gamma(\lambda) &:=\; \{ i \,|\, \mysub{\lambda}{i} = 0 <  \mysub{M\lambda + q}{i} \} \label{lcpind:gamma}.
	\end{align}\label{lcpind}
\end{subequations}	
If $\beta(\lambda) = \emptyset$, then $\lambda$ is called a \emph{strict complementarity solution} of \lcp{$q$}{$M$}.  Otherwise, it is called a \emph{nonstrict complementarity solution}.  We will suppress $\lambda$ when the dependence is clear from the context. 

A matrix $M$ is a Q-matrix if \lcp{$q$}{$M$} has a solution for all $q \in \R^{n_c}$.  A matrix $M$ is an R$_0$-matrix if \sol{$0$}{$M$} $= \{0\}$.  
A matrix $M$ is said to be a P-matrix if the principal minors are all positive, i.e.  
$\det(\mysub{M}{JJ}) > 0$ for all $J \subseteq \{1,\ldots,n_c\}$.  It is known that a P-matrix is both an R$_0$-matrix and a Q-matrix \cite[Theorem 3.9.22]{CottlePangStone}. The following lemma stating properties of solutions to \lcp{$q$}{$M$} will be useful in the remainder of the paper.

\begin{lemma}[\cite{CottlePangStone}]\label{lemma:Pmatprops}
Consider the \lcp{$q$}{$M$} where the matrix $M$ is a P-matrix.  Then
\begin{enumerate}[(a)]
	\item  \lcp{$q$}{$M$} has an unique solution $\lambda(q)$ for every $q$.
	\item $\lambda(q)$ is a piecewise linear function of $q$, is globally Lipschitz continuous, and is
	directionally differentiable.  
	\item There exists a constant $c_M > 0$ such that $\|\lambda(q)\| \leq c_M \|q\|$ for all $q \in \R^{n_c}$.
\end{enumerate}
\end{lemma}

\subsection{Lyapunov Stability}

Consider the discrete-time nonlinear dynamical system
\beq
	x_{k+1} \in f(x_k) ;\;\; x_0 = \hat{x}, \label{dsystem}
\eeq
where $f : \R^{n_x} \rightarrow 2^{\R^{n_x}}$ is a mapping into subsets of $\R^{n_x}$.  Let $x^e$ be an equilibrium point of the system~\eqref{dsystem}; i.e. $f(x^e) = 0$, and let $x_k(\hat{x})$, $k \in \Z_+$, be a trajectory of the system initialized at $\hat{x}$. We do not assume the trajectory is unique.

\begin{definition}
The equilibrium point $x^e$ is
\begin{enumerate}[(i)]
	\item stable (in the sense of Lyapunov) if for each $\epsilon > 0$, there is $\delta_\epsilon > 0$ such that 
	\[
		\| \hat{x} - x^e \| < \delta_\epsilon \implies \| x_k(\hat{x}) - x^e \| < \epsilon \quad \forall k \in \Z_+
	\] 
	for all trajectories $x_k(\hat{x})$; 
	\item exponentially stable if there exist scalars $\delta > 0$, $c > 0$ and $0 \leq \mu < 1$ such that
	\[
		\| \hat{x} - x^e \| < \delta \implies \| x_k(\hat{x}) - x^e \| \leq c \mu^k \|\hat{x} - x^e\| \quad \forall k \in \Z_+
	\]
	for all trajectories $x_k(\hat{x})$.
\end{enumerate}
\end{definition}

Note that the stability notions defined above are the so-called \emph{strong} exponential stability since they are expected to hold for every trajectory from the given initial condition.
If the function $f(x)$ is positively homogeneous in $x$, i.e. $f(\tau x) = \tau f(x)$ then the stability results hold \emph{globally} for all initial conditions $\hat{x} \in \R^{n_x}$.  Note that the functions defining the dynamics in both LCS and \dlcs\ are positively homogeneous.  We discuss how stability analysis for the homogeneous case can be extended to some inhomogeneous systems in \S ~\ref{sec:inhomdlcs}.

\subsection{Semidefiniteness and Matrix Copositivity}  

 Semidefiniteness with respect to a closed, not necessarily convex, cone ${\cal K}$ is called copositivity with respect to ${\cal K}$.  Specifically, a matrix $M \in \R^{n \times n}$ is said to be \emph{copositive} with respect to a cone ${\setK} \subseteq \R^n$ if $u^T M u \geq 0$ for all $u \in {\setK}$.  In this case, we say $M$ is a ${\cal K}$-copositive matrix, and denote this fact by $M \succcurlyeq_{\setK} 0$.  We use the notation $M \succcurlyeq_{\setK} \epsilon$ to mean that $u^T M u \geq \epsilon$ for all $u \in {\setK}$.  Similarly, if $u^T M u  > 0$ for all $u \in {\cal K}$ then $M$ is a strict ${\cal K}$-copositive matrix, which we denote by $M \succ_{\setK} 0$.  If $u^T M u  > \epsilon$ for all $u \in {\cal K}$, we say that $M \succ_{\setK} \epsilon$.  Note that the standard semidefinite and copositivity condtions are obtained by choosing $\setK = \R^n$ and $\setK = \R^n_+$ respectively.  For additional details, the reader is referred to~\cite{BundfussDurLyapunov2009,EichfelderPovh2013}.

\section{Stability of \dlcs\,}\label{sec:stability}

In this section we present stability conditions for the \dlcs, at $x^e = 0$.  We assume that $(x^e,\lambda^e) = (0,0)$ an equilibrium point of \dlcs, i.e. $(x^e,\lambda^e)$ is a solution of
\begin{subequations}
	\begin{align}
		x =&\; Ax + \ComplState\lambda \label{dlcseq:dyn} \\
		0 \leq \lambda \perp&\; \StateCompl x + \ComplCompl \lambda \geq 0. 
		\label{dlcseq:compl}
	\end{align}\label{dlcseq}
\end{subequations}
The uniqueness of the equilibrium is guaranteed if $\ComplCompl$ is a R$_0$-matrix i.e. \lcp{$0$}{$\ComplCompl$} $= \{0\}$. We first show sufficient conditions for a CQLF in \S\ref{sec:cqlf} and extend these to sufficient conditions for an EQLF in \S\ref{sec:eqlf}. A discussion on the applicability of stability results to inhomogeneous \dlcs\, is presented in \S\ref{sec:inhomdlcs} under the assumption that $\ComplCompl$ is a P-matrix.  

\subsection{Common Quadratic Lyapunov Function (CQLF)}\label{sec:cqlf}

We will first present sufficient conditions for stability of the equilibrium point using a quadratic Lyapunov function that depends only on the states, or a Common Quadratic Lyapunov function (CQLF)~\cite{Liberzon2003, Shorten2007}.

\begin{theorem}\label{thm:CQLF}
Consider the \dlcs\, in~\eqref{dlcs}.  Assume that $F$ is a Q-matrix and R$_0$-matrix.  Let $\psi : \R^{n_x+n_c} \rightarrow \R$ denote a quadratic function 
\begin{align}
	\psi(x,\lambda) &= \begin{pmatrix} x \\ \lambda \end{pmatrix}^T M(P_{xx}) \begin{pmatrix} x \\ \lambda \end{pmatrix}, \mbox{ where} \\ 
	M(P_{xx}) &= \begin{pmatrix} 
		A^TP_{xx} A - P_{xx} & A^TP_{xx}C \\
		C^TP_{xx}A & C^TP_{xx}C
	\end{pmatrix}, \label{defMmatrix}
\end{align}
and  $P_{xx} \in \Sy^{n_x}$.  The point $x^e = 0$ is 
\begin{enumerate}[(i)]
	\item \emph{stable} if there exists a $P_{xx} \succ 0$ with $\psi(x,\lambda) \leq 0$ for all $(x,\lambda) \in \gr\,\sol{\StateCompl x}{\ComplCompl}$, and 
	\item \emph{exponentially stable} if there exists a $P_{xx} \succ 0$ with $\psi(x,\lambda) < 0$ for all nonzero $(x,\lambda) \in \gr\,\sol{\StateCompl x}{\ComplCompl}$.
\end{enumerate}

\end{theorem}
\begin{proof}
Let $x_0$ be an arbitrary initial condition for the \dlcs\, and let $(x_k,\lambda_k)$ for $k \geq 1$ be a trajectory from such an initial state with $\lambda_{k+1} \in \sol{\StateCompl x_k}{\ComplCompl}$.  
Consider the function $V(x_k) = x_k^T P_{xx}x_k$.  By the positive definiteness of $P_{xx}$, $V(x) \geq 0$ for all $x \in \R^{n_x}$ and $V(x) = 0$ if and only if $x = x^e$.  Thus, $V(x)$ is a Lyapunov function. 
The Lyapunov function at $x_{k+1}$ is
\[
	V(x_{k+1}) = \begin{pmatrix} x_k \\ \lambda_{k+1} \end{pmatrix}^T 
	\begin{pmatrix} 
		A^TP_{xx} A & A^TP_{xx}C \\
		C^TP_{xx}A & C^TP_{xx}C
	\end{pmatrix} \begin{pmatrix} x_k \\ \lambda_{k+1} \end{pmatrix}.
\]
It can be verified that $\psi(x_k,\lambda_{k+1}) = V(x_{k+1}) - V(x_k)$.  If $\psi(x,\lambda) \leq 0$ for all $(x,\lambda) \in \gr\,\sol{\StateCompl x}{\ComplCompl}$ then $V(x_{k+1}) \leq V(x_k)$ for all $x_k \in \R^{n_x}$.  In other words, along all trajectories of \dlcs\, from $x_0$  
\begin{align}
    V(x_{k+1}) \leq V(x_k) \implies V(x_k) \leq V(x_0). \label{decV}
\end{align}
Since $P_{xx} \succ 0$ there exists $\mu > 0$ such that $V(x_k) \geq \mu \|x_k\|^2$.  Further, $V(x) \leq \|P_{xx}\| \|x\|^2$.  Combining the two observations in~\eqref{decV} yields $\|x_k\|^2 \leq (\|P_{xx}\|/\mu) \|x_0\|^2$ for all $k$ and all trajectories from $x_0$.  The claim on stability in (i) follows.  
If $\psi(x,\lambda) < 0$ then there exists a $\omega > 0$ such that $\psi(x,\lambda) \leq -\omega \|(x,\lambda)\|^2$ for all $(x,\lambda) \in \gr\,\sol{\StateCompl x}{\ComplCompl}$.  Such a $\omega$ can be obtained as $-\omega := \max \{ (x,\lambda)^T M(P_{xx}) (x,\lambda) \,|\, \lambda \in \sol{\StateCompl x}{\ComplCompl},\, \|(x,\lambda)\| = 1 \}$ which is well defined.
Then along all trajectories of \dlcs\,
\[
 V(x_{k+1}) - V(x_k) \leq -\omega \|(x_k,\lambda_{k+1})\|^2 \leq -\omega \|x_k\|^2 \leq -(\omega/\kappa) V(x_k)
\]
where the second inequality follows from $\|x_k\|^2 \leq \|(x_k,\lambda_{k+1})\|^2$ and the third from $V(x_k) \leq \kappa \|x_k\|^2$ where $\kappa = \max(2\omega, \|P_{xx}\|)$.  Hence, $V(x_k) \leq (1 - \omega/\kappa)^k V(x_0)$ along all trajectories of \dlcs\, from $x_0$.  Combining this with $V(x) \geq \omega \|x\|^2$ and $V(x_0) \leq \kappa \|x_0\|^2$ yields $\|x_k\|^2 \leq (\kappa/\mu)(1 - \omega/\kappa)^k \|x_0\|^2$ and the claim in (ii) holds.
\end{proof}

Theorem~\ref{thm:CQLF} shows that a sufficient condition for the exponential stability of the solution $x^e = 0$ to \eqref{dlcs} is to
\begin{subequations}
\begin{align}
	\text{Find a} &\;\; P_{xx} \in \Sy^{n_x} \\
	\text{ such that } &\; 	P_{xx} \succ 0, \mbox{ and } \label{feasPxx:setK1} \\
				&\; -M(P_{xx})  \succ_{\setK} 0, \label{feasPxx:setK2}
\end{align}\label{feasPxx}
\end{subequations}
where $M(P_{xx})$ is defined in~\eqref{defMmatrix}, and the set $\setK  \subseteq \R^{n_x+n_c}$ is defined as
\begin{equation}
{\setK} = \gr\,\sol{\StateCompl x}{\ComplCompl}. \label{defK}
\end{equation}

\subsection{Extended Quadratic Lyapunov Function (EQLF)}\label{sec:eqlf}

The use of a quadratic function in the states alone can limit the set of systems that can be stabilized.  Specifically, there exist linear switching systems that can be stabilized, but for which the CQLF does not exist~\cite[\S2.1.5]{Liberzon2003}.  This has motivated the use of unique quadratic Lyapunov functions for each mode of the switching system with continuity imposed across the switching boundaries~\cite{Branicky1998, DecarloBranicky2000, Hespanha2004, JohanssonRantzer1998, Liberzon2003, Shorten2007}.  As noted in the introduction, this construction can also be restrictive since the Lyapunov decrease condition for a mode is enforced over the entire state-space rather than over the subset of the state-space corresponding to the mode.  In a significant departure from prior literature, \c{C}amlibel, Pang and Shen~\cite{CamlibelPangShen2006} proposed a quadratic Lyapunov function that depends on $(x(t),\lambda(t))$ for the LCS~\eqref{lcs}.  This Lyapunov function is piecewise quadratic since $\lambda(t)$ is a piecewise linear function (see Lemma~\ref{lemma:Pmatprops}).  We follow a similar path for the \dlcs\, and derive sufficient conditions for stability of $(x^e,\lambda^e) = 0$.

\begin{theorem}\label{thm:EQLF}
Consider the \dlcs\, in~\eqref{dlcs}.  Assume that $F$ is a Q-matrix and an R$_0$-matrix.  Let $\widehat{\psi} : \R^{n_x+2n_c} \rightarrow \R$ denote a quadratic function 
\begin{align}
	\widehat{\psi}(x,\lambda,\hatlambda) &=\; \begin{pmatrix} x \\ \lambda \\ \hatlambda \end{pmatrix}^T \widehat{M}(P) \begin{pmatrix} x \\ \lambda \\ \hatlambda \end{pmatrix}, \\
	\text{with } \widehat{M}(P) &=\; \begin{pmatrix} 
		A^TP_{xx} A - P_{xx} & A^TP_{xx}\ComplState - P_{x\lambda} &  A^TP_{x\lambda} \\
		\ComplState^TP_{xx}A - P_{x\lambda}^T & \ComplState^TP_{xx}\ComplState - P_{\lambda\lambda} & \ComplState^TP_{x\lambda} \\
		 P_{x\lambda}^TA& P_{x\lambda}^T\ComplState & P_{\lambda\lambda}
	\end{pmatrix}, \label{defMhatmatrix}
\end{align}
where $ P = \begin{psmallmatrix} P_{xx} & P_{x\lambda} \\ P_{x\lambda}^T & P_{\lambda\lambda} \end{psmallmatrix}$, 
$P_{xx} \in \Sy^{n_x}$, $P_{\lambda\lambda} \in \Sy^{n_c}$ and $P_{x\lambda} \in \R^{n_x \times n_c}$.  If there exists $P_{xx}, P_{x\lambda}, P_{\lambda\lambda}$ such that the function $\widehat{V}(x,\lambda)$ 
\begin{equation}
	\widehat{V}(x,\lambda) = \begin{pmatrix} x \\ \lambda \end{pmatrix}^T 
	P
	\begin{pmatrix} x \\ \lambda \end{pmatrix}
\end{equation}
satisfies $\widehat{V}(x,\lambda) \geq 0$ for all $(x,\lambda) \in \gr\,\sol{\StateCompl x}{\ComplCompl}$ and $\widehat{V}(x,\lambda) = 0$ if and only if $(x,\lambda) = (x^e,\lambda^e)$,  then $(x^e,\lambda^e) = 0$ is
\begin{enumerate}[(i)]
	\item stable if $\widehat{\psi}(x,\lambda,\hatlambda) \leq 0$ for all $(x,\lambda) \in \gr\,\sol{\StateCompl x}{\ComplCompl}$ and $\hatlambda \in \sol{\StateCompl(A x + \ComplState \lambda)}{\ComplCompl}$, and
	\item exponentially stable if $\widehat{\psi}(x,\lambda,\hatlambda) < 0$ for all nonzero $(x,\lambda) \in \gr\,\sol{\StateCompl x}{\ComplCompl}$ and $\hatlambda \in \sol{\StateCompl(A x + \ComplState \lambda)}{\ComplCompl}$.
\end{enumerate}
\end{theorem}
\begin{proof}
Let $x_0$ be an arbitrary initial condition for the \dlcs\, and let $(x_k,\lambda_k)$ for $k \geq 1$ be a trajectory from such an initial state with $\lambda_{k+1} \in \sol{\StateCompl x_k}{\ComplCompl}$.  
Then 
	\begin{align*}
		\widehat{V}(x_{k+1},\lambda_{k+2}) 
		=&\; \begin{pmatrix} x_{k+1} \\ \lambda_{k+2} \end{pmatrix}^T 
			   	\begin{pmatrix} P_{xx} & P_{x\lambda} \\ P_{x\lambda}^T & P_{\lambda\lambda} \end{pmatrix}
				\begin{pmatrix} x_{k+1} \\ \lambda_{k+2} \end{pmatrix} \\
			=&\; \begin{pmatrix} Ax_k + \ComplState \lambda_{k+1} \\ \lambda_{k+2} \end{pmatrix}^T 
			   	\begin{pmatrix} P_{xx} & P_{x\lambda} \\ P_{x\lambda}^T & P_{\lambda\lambda} \end{pmatrix}
				\begin{pmatrix} Ax_k + \ComplState \lambda_{k+1} \\ \lambda_{k+2} \end{pmatrix} \\
			=&\; \begin{pmatrix} x_k \\ \lambda_{k+1} \\ \lambda_{k+2} \end{pmatrix}^T
			\begin{pmatrix}
				A^TP_{xx}A & A^TP_{xx}\ComplState  & A^TP_{x\lambda} \\
				\ComplState^TP_{xx}A  & \ComplState^TP_{xx}\ComplState & \ComplState^TP_{x\lambda}  \\
				P_{x\lambda}^TA & P_{x\lambda}^T\ComplState & P_{\lambda\lambda}
			\end{pmatrix}
			\begin{pmatrix} x_k \\ \lambda_{k+1} \\ \lambda_{k+2} \end{pmatrix}.
	\end{align*}
It can be verified that $\widehat{V}(x_{k+1},\lambda_{k+2}) - \widehat{V}(x_k,\lambda_{k+1}) = \widehat{\psi}(x_k,\lambda_{k+1},\lambda_{k+2})$.   
The assumptions on $\widehat{V}$ ensure that $\widehat{V}(x_k,\lambda_{k+1}) \geq 0 \mbox{ and } \widehat{V}(x_{k+1},\lambda_{k+2}) \geq 0$.  
If the condition in (i) holds then 
\begin{align}
\widehat{V}(x_{k+1},\lambda_{k+2}) \leq \widehat{V}(x_k,\lambda_{k+1})
\implies 
\widehat{V}(x_{k},\lambda_{k+1}) \leq \widehat{V}(x_0,\lambda_1) 
\label{decVhat}
\end{align}
for all $k$.  Further, since $\widehat{V}(x_k,\lambda_{k+1}) > 0$ for all nonzero $(x_k,\lambda_{k+1}) \in \gr\sol{\StateCompl x_k}{\ComplCompl}$ we have that $\widehat{V}(x_k,\lambda_{k+1}) \geq \mu \|(x_k,\lambda_{k+1})\|^2$ for some $\mu > 0$.   Such a $\mu$ can be obtained as $\mu = \min \{ (x,\lambda)^T \widehat{M}(P) (x,\lambda) \,|\, (x,\lambda) \in \gr\sol{\StateCompl x}{\ComplCompl}, \|(x,\lambda)\| = 1 \}$.
Further, $V(x_0,\lambda_1) \leq \|P\| \|(x_0,\lambda_1)\|^2$.  Applying these two in~\eqref{decVhat} yields $\|(x_k,\lambda_{k+1})\|^2 \leq (\|P\|/\mu)\|(x_0,\lambda_1)\|^2$ and the claim on stability in (i) follows.  
If the condition in (ii) holds, then there exists $\omega > 0$ such that $\widehat{\psi}(x_k,\lambda_{k+1},\lambda_{k+2}) \leq -\omega \|(x_k,\lambda_{k+1},\lambda_{k+2})\|^2$.  Such a $\omega$ can be obtained as $-\omega := \max \{ (x,\lambda,\widehat{\lambda})^T \widehat{M}(P) (x,\lambda,\widehat{\lambda}) \,|\, \lambda \in \sol{\StateCompl x}{\ComplCompl},\, \widehat{\lambda} \in \sol{\StateCompl(A x + \ComplState \lambda)}{\ComplCompl},\, \|(x,\lambda,\widehat{\lambda})\| = 1 \}$ which is well defined. Then  
\begin{align*}
\widehat{V}(x_{k+1},\lambda_{k+2}) - \widehat{V}(x_k,\lambda_{k+1}) 
& \leq -\omega \|(x_k,\lambda_{k+1},\lambda_{k+2})\|^2 \leq -\omega \|(x_k,\lambda_{k+1})\|^2 \\
& \leq -(\omega/\kappa) \widehat{V}(x_k,\lambda_{k+1})
\end{align*}
where the second inequality follows from $\|(x_k,\lambda_{k+1})\|^2 \leq \|(x_k,\lambda_{k+1},\lambda_{k+2})\|^2$ and the third inequality from $\widehat{V}(x_k,\lambda_{k+1}) \leq \kappa \|(x_k,\lambda_{k+1})\|^2$, where $\kappa = \max(2\omega,\|P\|)$.  Thus, $\widehat{V}(x_k,\lambda_{k+1})$ $\leq (1 - \omega/\kappa)^k$ $\widehat{V}(x_0,\lambda_1)$ along all trajectories of \dlcs\, from the initial point $x_0$.  Combining with $\widehat{V}(x_k,\lambda_{k+1}) \geq \mu \|(x_k,\lambda_{k+1}\|^2$ yields 
\[
\|(x_k,\lambda_{k+1})\|^2 \leq (\kappa/\mu) (1-\omega/\kappa)^k
\|(x_0,\lambda_1)\|^2
\]
and the claim in (iii) follows.
\end{proof}

Note that setting $P_{x\lambda} = 0$ and $P_{\lambda\lambda} = 0$ in~\eqref{defMhatmatrix}, we obtain the sufficient conditions for stability with a CQLF.  
Theorem~\ref{thm:EQLF} shows that a sufficient condition for the exponential stability of the point $x^e = 0$ using an EQLF is to 
\begin{subequations}
\begin{align}
	\text{Find a} &\;\; P \in \Sy^{n_x+n_c} \\
	\text{ such that } &\; 	P \succ_{\setK} 0,  \mbox{ and } \label{feasP:setK1} \\
				&\; -\widehat{M}(P)  \succ_{\widehat{\setK}} 0, \label{feasP:setK2}
\end{align}\label{feasP}
\end{subequations}
where $\widehat{M}(P)$ is defined in~\eqref{defMhatmatrix}, and the set $\widehat{\setK} \subseteq \R^{n_x+2n_c}$ is defined as 
\begin{equation}
\begin{aligned}
\widehat{{\cal K}} =&\; \{ (x,\lambda,\hatlambda) \,|\, (x,\lambda) \in \gr \,\sol{\StateCompl x}{\ComplCompl},\, \hatlambda \in \sol{\StateCompl(A x + \ComplState \lambda)}{\ComplCompl} \}.
\end{aligned}  \label{defKhat}
\end{equation}

\subsection{Extension to Inhomogeneous \dlcs}\label{sec:inhomdlcs}


In this section, we consider a generalization of the \dlcs, the autonomous inhomogeneous \dlcs\, which is given as 
\begin{subequations}
	\begin{align}
		x_{k+1} =&\; Ax_k + \ComplState\lambda_{k+1} + f\label{dlcsinh:dyn}\\
		0 \leq \lambda_{k+1} \perp&\; \StateCompl x_{k} + \ComplCompl 
		\lambda_{k+1} + g \geq 0, \label{dlcsinh:compl}
	\end{align}\label{dlcsinh}
\end{subequations}
where $f \in \R^{n_x}$ and $g \in \R^{n_c}$.

In this section, we show how the stability results in~\S\ref{sec:cqlf} and \S\ref{sec:eqlf} can be applied to establish the local stability of an inhomogeneous \dlcs~\eqref{dlcsinh}. For the results in this section, we must make the stronger assumption that  $\ComplCompl$ is a P-matrix. 

A point $(x^e,\lambda^e)$ is said to be an equilibrium of~\eqref{dlcsinh} if 
\begin{subequations}
	\begin{align}
		x^e =&\; Ax^e + \ComplState\lambda^e + f\label{dlcsinheq:dyn}\\
		0 \leq \lambda^e \perp&\; \StateCompl x^e + \ComplCompl 
		\lambda^e + g \geq 0. \label{dlcsinheq:compl}
	\end{align}\label{dlcsinheq}
\end{subequations}
Let $(\alpha^e,\beta^e,\gamma^e)$ denote the index partition according to~\eqref{lcpind} at the equilibrium $x^e$.  Suppose that $\ComplCompl$ is a P-matrix, and let $\lambda(x)$ be the unique solution of the \lcp{$\StateCompl x + h$}{$\ComplCompl$}.  From the global Lipschitz continuity of $\lambda(x)$ (Lemma~\ref{lemma:Pmatprops}(b)), we have that there exists a neighborhood ${\cal N}(x^e)$ around $x^e$ such that for all $x \in {\cal N}(x^e)$, $\mysub{\lambda(x)}{i} > 0$ for all $i \in \alpha^e$ and $\mysub{\StateCompl x + \ComplCompl \lambda(x)}{i} > 0$ for all $i \in \gamma^e$.  Hence, $\alpha(\lambda(x)) \supseteq \alpha^e$ and $\gamma(\lambda(x)) \supset \gamma^e$ holds for all $x \in {\cal N}(x^e)$. 

 Thus, for all $x \in {\cal N}(x^e)$ the $\lcp{\StateCompl x}{\ComplCompl}$ can be restated as 
\begin{subequations}
	\begin{align}
		0 &=\;  \mysub{\StateCompl x + \ComplCompl \lambda(x) + g}{\alpha^e} \\
		0 \leq \mysub{\lambda(x)}{\beta^e} &\perp\; \mysub{\StateCompl x + \ComplCompl \lambda(x) + g}{\beta^e} \geq 0 \\
		 \mysub{\lambda(x)}{\gamma^e} &=\; 0.
	\end{align}\label{localindslam}
\end{subequations}
Introducing $\delta x_{k+1} = x_{k+1} - x^e$, $\delta \lambda_{k+1} = \lambda_{k+1} - \lambda^e$, and using the observation in~\eqref{localindslam}, the \dlcs\, in~\eqref{dlcsinh} can be written for all $x \in {\cal N}(x^e)$ as
\begin{subequations}
	\begin{align}
		\delta x_{k+1} =&\; A \delta x_k + \ComplState \delta \lambda_{k+1} \label{dlcsinh1:dyn}\\
		0 =&\; \mysub{\StateCompl \delta x_k + \ComplCompl \delta \lambda_{k+1}}{\alpha^e} \\
		0 \leq \mysub{\delta \lambda_{k+1}}{\beta^e} \perp&\; \mysub{\StateCompl \delta x_{k} + \ComplCompl 
		\delta \lambda_{k+1}}{\beta^e} \geq 0 \label{dlcsinh1:compl} \\
		0 =&\; \mysub{\delta \lambda_{k+1}}{\gamma^e}.
	\end{align}\label{dlcsinh1}
\end{subequations}
The system~\eqref{dlcsinh1} can be further simplified to the homogeneous \dlcs\
\begin{subequations}
	\begin{align}
		\delta x_{k+1} =&\; \widehat{A}  \delta x_k + \widehat{\ComplState} \mysub{\delta \lambda_{k+1}}{\beta^e} 
		\label{dlcsinh2:dyn}\\
		0 \leq \mysub{\delta \lambda_{k+1}}{\beta^e} \perp&\; \widehat{\StateCompl} \delta x_{k} + \widehat{\ComplCompl}  
		\mysub{\delta \lambda_{k+1}}{\beta^e} \geq 0, \label{dlcsinh2:compl} 
	\end{align}\label{dlcsinh2}
\end{subequations}
where $\widehat{A} = (A - \mysub{\ComplState}{\bullet \alpha^e} (\mysub{\ComplCompl}{\alpha^e\alpha^e})^{-1} 
		\mysub{\StateCompl}{\alpha^e\bullet}$, $\widehat{\ComplState} = \mysub{\ComplState}{\bullet \beta^e} - 
		\mysub{\ComplState}{\bullet \alpha^e} 
		(\mysub{\ComplCompl}{\alpha^e\alpha^e})^{-1} \mysub{\ComplCompl}{\alpha^e\beta^e}$, 
		$\widehat{\StateCompl} = \mysub{\StateCompl}{\beta^e\bullet} - \mysub{\ComplCompl}{\beta^e\alpha^e} 
		(\mysub{\ComplCompl}{\alpha^e\alpha^e})^{-1} \mysub{\StateCompl}{\alpha^e\bullet}$, and 
		$\widehat{\ComplCompl} = \mysub{\ComplCompl}{\beta^e\beta^e} - \mysub{\ComplCompl}{\beta^e\alpha^e} 
		(\mysub{\ComplCompl}{\alpha^e\alpha^e})^{-1} \mysub{\ComplCompl}{\alpha^e\beta^e}$.  Thus, the stability analysis of the previous sections can be applied to the homogenous \dlcs\, in~\eqref{dlcsinh2} to verify local stability of $x^e$.

\section{A Cutting-Plane Algorithm for Computing the Lyapunov Function}\label{sec:algorithm}

The sufficient conditions \eqref{feasPxx} and \eqref{feasP} for exponential stability using a CQLF or an EQLF can be written in a general fashion as
\begin{subequations}
\begin{align}
	\text{Find} &\; \bfP \in \Sy^{n}  \\
	\text{ such that} &\; 	\bfP \succ_{\setK_1} 0, \label{feasPgen:setK1} \\
				&\; -\bfM(\bfP)  \succ_{\setK_2} 0 \label{feasPgen:setK2}
\end{align}\label{feasPgen}
\end{subequations}
where $\setK_1$ and $\setK_2$ are closed (not necessarily convex) cones and the matrix $\bfM \in \Sy^m$ depends on the matrix $\bfP$.  The feasibility conditions require that $\bfP$ is strict $\setK_1$-copositive and $-\bfM$ is strict $\setK_2$-copositive.

A typical approach to handling the strict $\setK_1$-copositivity and strict $\setK_2$-copositivity is to relax the conditions into linear matrix inequalities using the S-Lemma~\cite{Aydinoglu2020,BoydLMI}.  The transformation to such a semidefinite formulation for the CQLF and EQLF is provided in Appendix~\ref{sec:appendix}.  Our intent is to provide an exact algorithm which we describe next.

\subsection{Cutting Plane Algorithm}\label{sec:cuttingplane}

We state by recalling a result from~\cite{BundfussDurLyapunov2009} that bounds the set of test vectors for copositivity with respect to a cone.  
\begin{lemma}\cite[Lemma 1]{BundfussDurLyapunov2009}\label{lemma:normBall}
Let $\| \cdot \|_1$ be the 1-norm on $\R^{n}$ and $\setC \subset \R^n$ be a closed cone.  Then
\[
	N \succcurlyeq_{{\cal C}} 0 \iff z^TNz \geq 0 \text{ for all } z \in {\cal C}, \|z\|_1 = 1.
\]
An analogous results holds for $N \succ_{\cal C} 0$.
\end{lemma}

Thus, the feasibility problem~\eqref{feasPgen} can be equivalently cast as the semi-infinite optimization problem
\begin{subequations} 
\begin{align}
	\max\limits_{\mu,\bfP \in \Sy^{n}}	&\;	\mu 		\label{sicpPgen:obj}\\
	\text{s.t.}							&\;	u^T\bfP u \,\,\,\,\geq \mu \,\forall\, u \in \setU \label{sicpPgen:c1} \\
									&\;	-v^T \bfM(\bfP) v \geq \mu	\,\forall\, v \in \setV \label{sicpPgen:c2} 
\end{align}\label{sicpPgen}
\end{subequations}
where $\setU = \setK_1 \cap \Sphere^n_1$ and $\setV = \setK_2 \cap \Sphere^m_1$.   If the optimal objective value of~\eqref{sicpPgen} is positive then the origin is exponentially stable for~\eqref{dlcs}.  The presence of infinite number of constraints poses a difficulty for computation. We propose to solve~\eqref{sicpPgen} using a cutting plane algorithm which alternates between a \emph{master problem} and two \emph{separation problems}.  At an iteration $l$ of the cutting plane algorithm the master problem attempts to find a matrix $\bfP^l$ for which the inequalities~\eqref{sicpPgen:c1}-\eqref{sicpPgen:c2} hold over finite sets $\setU^l, \setV^l$. The separation problem verifies if the computed $\bfP^l$ satisfies the inequalities in~\eqref{sicpPgen:c1}-\eqref{sicpPgen:c2} for all vectors in $\setU$ and $\setV$.  If satisfied, then $\bfP^l$ is a matrix for which~\eqref{feasPgen} holds.  If not, then the separation problem identifies points $u^l, v^l$ for which the inequalities~\eqref{feasPgen:setK1}-\eqref{feasPgen:setK2} fail to hold. These points are appended to the sets $\setU^l, \setV^l$, and the algorithm continues.

The \emph{master problem} associated with~\eqref{sicpPgen} is
\begin{subequations}
\begin{align}
	\max\limits_{\mu,\bfP \in \Sy^{n}}	&\;	\mu 		\label{sicpPgen-mp:obj} \\
	\text{s.t.}							&\;	u^T\bfP u \geq \mu \,\forall\, u \in \setU^l \label{sicpPgen-mp:c1} \\
									&\;	-v^T \bfM(\bfP) v \geq \mu	\,\forall\, v \in \setV^l \label{sicpPgen-mp:c2} \\
									&\;	-1 \leq \mysub{\bfP}{ij} \leq 1 \,\forall\, i,j = 1,\ldots,n.	\label{sicpPgen-mp:c3}
\end{align}\label{sicpPgen-mp}
\end{subequations}
The constraint~\eqref{sicpPgen-mp:c3} is a normalization constraint and does not affect the feasible region due to the homogeneity of~\eqref{feasPgen}.  Specifically, if $\hat{\bfP}$ satisfies~\eqref{feasPgen} then $\hat{\bfP}/(\max_{i,j} \mysub{\hat{\bfP}}{ij})$ also satisfies~\eqref{feasPgen}.  From a computational standpoint~\eqref{sicpPgen-mp:c3} serves to bound the feasible region and ensures that an optimal solution to the linear program~\eqref{sicpPgen-mp} always exists as long as $\setU^l \neq \emptyset$ as we show in the following.  
Note that from ~\eqref{sicpPgen-mp} we have that 
\begin{align}
	\mu = \max_{\bfP :  |\mysub{\bfP}{ij}| \leq 1} \min \left\{ \min_{u \in \setU^{l}} u^T\bfP u ,
	 \min_{v \in \setV^{l}} -v^T\bfM(\bfP) v
	\right\}. \label{mubndbyminU}
\end{align}
Thus, if $\setU^l \neq \emptyset$, an optimal solution to~\eqref{sicpPgen-mp} must have optimal value at most $n$, since 
\[\begin{aligned}
	\mu \leq&\; 	\max_{\bfP :  |\mysub{\bfP}{ij}| \leq 1} \max_{u \in \setU^{l}} u^T\bfP u 
		\leq	\max_{\bfP :  |\mysub{\bfP}{ij}| \leq 1 , u \in \Sphere^{n}_1} u^T\bfP u 
		\leq \max_{ u \in \Sphere^{n}_1} \sum\limits_{ij} u_iu_j 
		\leq n.
\end{aligned}\]
Since $\bfP = 0, \mu = 0$ is a feasible solution to~\eqref{sicpPgen-mp}, the master linear program must have an optimal solution, which we denote as $(\mu^l,\bfP^l)$.

The \emph{separation problem} associated with~\eqref{sicpPgen-mp} for the cone $\setK_1$ (constraints~\eqref{sicpPgen-mp:c1}) is 
\begin{subequations}
\begin{align}
	\min\limits_{\nu_1\in \R, u \in \R^n}	
							&\;	\nu_1 		\label{sicpPgen-spu:obj} \\
	\text{s.t.}					&\;	u^T\bfP^l u \leq \nu_1 \label{sicpPgen-spu:c1} \\
							&\;	u \in \setK_1 \cap \Sphere^n_1. \label{sicpPgen-spu:c2}
\end{align}\label{sicpPgen-spu}
\end{subequations}
The \emph{separation problem} associated with~\eqref{sicpPgen-mp} for the cone $\setK_2$ (constraints~\eqref{sicpPgen-mp:c1}) is
\begin{subequations}
\begin{align}
	\min\limits_{\nu_2 \in \R, v \in \R^m}	
							&\;	\nu_2 		\label{sicpPgen-spv:obj} \\
	\text{s.t.}					&\;	-v^T\bfM(\bfP^l) v \leq \nu_2 \label{sicpPgen-spv:c1} \\
							&\;	v \in \setK_2 \cap \Sphere^m_1. \label{sicpPgen-spv:c2}
\end{align}\label{sicpPgen-spv}
\end{subequations}
In the separation problems in~\eqref{sicpPgen-spu} and~\eqref{sicpPgen-spv} we have included the $\setK_1 \cap \Sphere^n_1$ and $\setK_2 \cap \Sphere^m_1$ constraint without specifying a formulation. We provide a mixed integer programming formulation for the CQLF conditions~\eqref{feasPxx} in \S~\ref{sec:CQLFalgorithm} and for the EQLF conditions~\eqref{feasP} in \S ~\ref{sec:EQLFalgorithm}.

Algorithm~\ref{algo:CPPgen} summarizes the steps in the cutting plane algorithm.  The steps of solving the master and separation problems at each iteration of the algorithm are described in Lines~\ref{algo:l5}-\ref{algo:l15}.   The algorithm aims to find $\bfP$, $u \in \Sphere^n_1$, $v \in \Sphere^m_1$ satisfying
\begin{equation}
u^T \bfP u \geq \epsilon \ \forall u \in \setK_1, \ -v^T \bfM(\bfP) v \geq \epsilon \ \forall v \in \setK_2, \ \eqref{sicpPgen-mp:c3}    \label{feasPgen-eps}
\end{equation}
for an appropriately small positive value $\epsilon$.
Because the inequality system \eqref{feasPgen} is homogeneous, if the strict inequality system
\begin{equation}
\bfP \succ_{\setK_1} 0,\, -\bfM(\bfP) \succ_{\setK_2} 0,\, \eqref{sicpPgen-mp:c3}    \label{feasPgen-0}
\end{equation}
has a solution $\bfP$, then by scaling the vectors $u$ and $v$, the value of the positive quadratic forms $u^T \bfP u$ and $-v^T \bfM(\bfP) v$ can be made arbitrarily close to zero, and thus for numerical computing purposes, we use a small positive tolerance $\epsilon$. 

\begin{algorithm}
\caption{A cutting plane algorithm to find $\bfP$ satisfying~\eqref{feasPgen}}
\label{algo:CPPgen}
\begin{algorithmic}[1]
\STATE \textbf{Input}: $\epsilon > 0$, MaxIter. \label{algo:l1}
\STATE \textbf{Output}: $\bfP$ satisfying~\eqref{feasPgen} if one exists. \label{algo:l2}
\STATE Set $l = 0$. Initialize the sets $\setU^0 \subset \setU$ and $\setV^0 \subset \setV$. \label{algo:l3}
\WHILE{$l < \text{MaxIter}$} \label{algo:l4}
\STATE Solve~\eqref{sicpPgen-mp} to obtain optimal solution $(\mu^l,\bfP^l)$. \label{algo:l5}
\IF{$\mu^l < \epsilon$} \label{algo:l6}
	\STATE No solution exists satisfying user-specified tolerances. Terminate. \label{algo:l7}
\ENDIF
\STATE Solve~\eqref{sicpPgen-spu} and~\eqref{sicpPgen-spv} using $\bfP^l$ to obtain optimal solutions $(\nu_1^l,u^l)$ and $(\nu_2^l,v^l)$ respectively. \label{algo:l8}
\IF{$\min(\nu_1^l,\nu_2^l) \geq \epsilon$} \label{algo:l9}
	\STATE Return $\bfP^l$ as a feasible solution. Terminate. \label{algo:l10}
\ENDIF 
\IF{$\nu_1^l < \epsilon$} \label{algo:l11}
	\STATE $\setU^{l+1} \leftarrow \setU^l \cup \{u^l\}$ \label{algo:l12}
\ENDIF
\IF{$\nu_2^l < \epsilon$} \label{algo:l13}
	\STATE $\setV^{l+1} \leftarrow \setV^l \cup \{v^l\}$ \label{algo:l14}
\ENDIF \label{algo:l15}
\STATE Set $l \leftarrow l + 1$.
\ENDWHILE
\end{algorithmic}
\end{algorithm}

\noindent It is easy to see that if on solving~\eqref{sicpPgen-mp} (in Line~\ref{algo:l5}) $\mu^ l < \epsilon$ then~\eqref{sicpPgen} is infeasible.

\begin{lemma}\label{lemma:sicpPgenInfeasible}
Suppose $\mu^l < \epsilon$ at some iteration $l$ of Algorithm~\ref{algo:CPPgen}.  Then there exists no $\bfP$ in~\eqref{sicpPgen-mp:c3} satisfying~\eqref{feasPgen-eps}.
\end{lemma}  
\begin{proof}
If $\mu^l < \epsilon$ then by~\eqref{mubndbyminU} 
\begin{align*}
	&\; \max_{\bfP :  |\mysub{\bfP}{ij}| \leq 1} \min \left\{ \min_{u \in \setU^{l}} u^T\bfP u ,
	 \min_{v \in \setV^{l}} -v^T\bfM(\bfP) v
	\right\} < \epsilon \\
	\implies &\; \max_{\bfP :  |\mysub{\bfP}{ij}| \leq 1} \min \left\{ \min_{u \in \setU} u^T\bfP u ,
	 \min_{v \in \setV} -v^T\bfM(\bfP) 
	\right\} < \epsilon
\end{align*}
where the second inequality follows by noting that the minimization is over larger sets $\setU^l \subseteq \setU$ and $\setV^l \subseteq \setV$.  Hence, there exists no $\bfP$ satisfying~\eqref{feasPgen-eps}, which proves the claim.
\end{proof}

Thus, if the Algorithm~\ref{algo:CPPgen} terminates before hitting the maximum iteration limit, it either terminates with a 
$\bfP \in \Sy^{n}$ such that $|\mysub{\bfP}{ij}| \leq 1$, $\bfP \succcurlyeq_{\setK_1} \epsilon$ and $-\bfM(\bfP)  \succcurlyeq_{\setK_2} \epsilon$, or a proof that no such matrix $\bfP$ exists.  If the iteration limit is reached, no conclusions on the stability of the system can be drawn.

\subsection{Common Quadratic Lyapunov Function}\label{sec:CQLFalgorithm}

To state the sufficiency conditions \eqref{feasPxx} required for a CQLF in the general form~\eqref{feasPgen}, we set $n = n_x$, $\bfP = P_{xx}$, $\setK_1 = \R^{n_x}$, $m = (n_x+n_c)$, $\bfM = M(P_{xx})$ in~\eqref{defMmatrix}, and $\setK_2 = \setK = \gr\,\sol{\StateCompl x}{\ComplCompl} \subseteq \R^{n_x+n_c}$.  
In the separation problem for a CQLF, we are given a matrix $P_{xx}^l$, and we wish to verify if the conditions \eqref{feasPxx} hold for this matrix.  Thus, we seek a vector $w = (x,\lambda) \in \Sphere^{n_x+n_c}_1$ such that either 
$w^T P_{xx}^l w \leq 0$ or $-w^T M(P_{xx}^l) w \leq 0$ for some $(x,\lambda) \in \gr\,\sol{\StateCompl x}{\ComplCompl}$.  If no such vector exits, then we have verified that $P_{xx}^l$ satisfies conditions~\eqref{feasPxx}, else the associated vector may be added to the appropriate set $\setU^l$ or $\setV^l$ in the master problem~\eqref{sicpPgen-mp}.

The separation problem for the constraint $P_{xx} \succ 0$ 
can be posed as the following mixed integer quadratic program, which minimizes the value of the quadratic form $x^T P_{xx}^l x$ over a 1-norm ball:
\begin{subequations}
\begin{align}
    \min\limits_{x,y, x^+,x^-,\nu} &\;\nu \label{eq:cqlf-s-obj}\\
    \text{s.t}  \qquad                    x^T P_{xx}^l x &\leq \nu \label{eq:cqlf-s-obj2}\\
      x &= x^+ - x^- \label{eq:cqlf-s-1-split}\\
     \mathbf{1}^T(x^+ + x^- ) &= 1 \label{eq:cqlf-s-1norm}\\
     0 \leq x^+ &\leq y \label{eq:cqlf-s-splitbin1}\\
     0 \leq x^- &\leq \mathbf{1} - y \label{eq:cqlf-s-splitbin2}\\
     y &\in \{0,1\}^{n_x} \label{eq:cqlf-s-bin1},
\end{align}
\label{eq:cqlf-sep-1}
\end{subequations}
where $\mathbf{1}$ is a vector of all ones.  

The separation problem for the constraint $-M(P_{xx}^l) \succcurlyeq_{\setK} 0$ can be posed as the following mixed integer quadratic program in variables $w=(x,\lambda) \in \mathbb{R}^{n_x+n_c}$, binary variables $y \in \{0,1\}^{n_x}$, $z \in \{0,1\}^{n_c}$ and auxiliary variables $x^+, x^- \in \mathbb{R}^{n_x}$:
\begin{subequations}
\begin{align}
    \min\limits_{x,\lambda,y, x^+,x^-,\nu} &\;\nu \label{eq:cqlf-obj}\\
    \text{s.t}  \qquad                    - (x,\lambda)^T M(P_{xx}^l) (x,\lambda) &\leq \nu \label{eq:cqlf-obj2}\\
      x &= x^+ - x^- \label{eq:cqlf-split}\\
     \mathbf{1}^T(x^+ + x^- + \lambda) &= 1 \label{eq:cqlf-1norm}\\
     0 \leq x^+ &\leq y \label{eq:cqlf-splitbin1}\\
     0 \leq x^- &\leq \mathbf{1} - y \label{eq:cqlf-splitbin2}\\
     y &\in \{0,1\}^{n_x} \label{eq:cqlf-bin1}\\
     0 \leq \lambda &\leq z \label{eq:cqlfk1}\\
     0 \leq \sum_{j=1}^{n_x} \mysub{D}{ij} \mysub{x}{j} + \sum_{j=1}^{n_c} \mysub{F}{ij} \mysub{\lambda}{j}  &\leq \theta_i (1-\mysub{z}{i}) \quad \forall i=1,\ldots,n_c \label{eq:cqlfk2}\\
     z &\in \{0,1\}^{n_c} \label{eq:cqlf-bin2},
\end{align}
\label{eq:cqlf-sep-2}
\end{subequations}
where ${\bf 1}$ is a vector of all ones of conformal dimension, and the constant
\[ \theta_i = \max  \left( \max_j \left| \mysub{D}{ij} \right| , \max_k \left( \mysub{F}{ik} \right) \right) \]
provides an upper bound on the value of the left hand side of constraint~\eqref{eq:cqlfk2} for any $(x,\lambda) \in \Sphere^{n_x+n_c}_1$.
The objective~\eqref{eq:cqlf-obj} and constraint~\eqref{eq:cqlf-obj2} seek to minimize the value of the (possibly nonconvex) quadratic function $(x,\lambda)^T M(P_{xx}^l) (x,\lambda)$.  Constraints~\eqref{eq:cqlf-split}---\eqref{eq:cqlf-bin1} ensure that $(x,\lambda) \in \Sphere^{n_x+n_c}_1$, and constraints~\eqref{eq:cqlfk1}---\eqref{eq:cqlf-bin2} enforce the conditions that $(x,\lambda) \in \gr\,\sol{\StateCompl x}{\ComplCompl}$.

\subsection{Extended Quadratic Lyapunov Function}\label{sec:EQLFalgorithm}

To state the sufficiency conditions for the EQLF in Theorem~\ref{thm:EQLF} in terms of the general setting of~\eqref{feasPgen}, we can set $n = (n_x+n_c)$, $\bfP = P$, $\setK_1 = \setK$, $m = (n_x+2n_c)$, $\bfM = \widehat{M}(P)$, and $\setK_2 = \widehat{\setK}$ where $\widehat{{\cal K}} = \{ (x,\lambda,\hatlambda) \,|\, (x,\lambda) \in \gr \sol{\StateCompl x}{\ComplCompl},\, \hatlambda \in \sol{\StateCompl(A x + \ComplState \lambda)}{\ComplCompl} \}$.  In the separation problem for an EQLF, we are given a matrix $P^l$, and we seek a vector $v = (x,\lambda,\hat{\lambda}) \in \mathbb{B}^{n_x + 2 n_c}_1$ such that either
\begin{align*}
    (i) \quad (x,\lambda)^T P^l (x,\lambda) < 0 & \mbox{ for some } (x,\lambda) \in \gr\,\sol{\StateCompl x}{\ComplCompl}; \mbox{ or }\\
    (ii) - (x,\lambda,\hat{\lambda})^T \widehat{M}(P^l) (x,\lambda,\hat{\lambda}) < 0 & \text{ for some } (x,\lambda) \in \gr \,\sol{\StateCompl x}{\ComplCompl} \\
    & \text{ and } \hatlambda \in \sol{\StateCompl(A x + \ComplState \lambda)}{\ComplCompl}.  
\end{align*}

The separation problem for the condition (i) ($P \succ_{\setK} 0$) is very similar to \eqref{eq:cqlf-sep-2} and thus not repeated here.  The separation problem for condition (ii) ($\widehat{M}(P) \succ_{\widehat{\setK}} 0$) can be posed as the following quadratic mixed integer program:
\begin{subequations}
\begin{align}
 \min\limits_{x,\lambda,\hat{\lambda},y,z,w,x^+,x^-,\nu} & \; \nu \label{eq:eqlf-obj}\\
  \text{s.t}  \qquad      (x,\lambda,\hat{\lambda})^T \widehat{M}(P^l) (x,\lambda,\hat{\lambda}) &\leq \nu \label{eq:eqlf-obj2}\\
      x &= x^+ - x^- \label{eq:eqlf-split}\\
     \mathbf{1}^T(x^+ + x^- + \lambda + \hat{\lambda}) &= 1 \label{eq:eqlf-1norm}\\
     0 \leq x^+ &\leq y \label{eq:eqlf-splitbin1}\\
    0 \leq x^- &\leq \mathbf{1} - y \label{eq:eqlf-splitbin2}\\
     y &\in \{0,1\}^{n_x} \label{eq:eqlf-bin1}\\
     0 \leq \lambda &\leq z \label{eq:eqlfk1}\\
     0 \leq \sum_{j=1}^{n_x} \mysub{D}{ij} \mysub{x}{j} + \sum_{j=1}^{n_c} \mysub{F}{ij} \mysub{\lambda}{j}  &\leq \theta_i (1-\mysub{z}{i}) \quad \forall i=1,\ldots,n_c \label{eq:eqlfk2}\\
     z &\in \{0,1\}^{n_c} \label{eq:eqlf-bin2}\\
     0 \leq \hat{\lambda} &\leq w \label{eq:eqlf-lamhat1}\\
     0 \leq  \sum_{j=1}^{n_x} \mysub{DA}{ij} \mysub{x}{j} + \sum_{j=1}^{n_c} \mysub{DC}{ij} \mysub{\lambda}{j} + \sum_{j=1}^{n_c} \mysub{F}{ij} \mysub{\hat{\lambda}}{j} &\leq \Theta_i (1 - \mysub{w}{i}) \quad \forall i=1,\ldots,n_c \label{eq:eqlf-lamhat2}\\
     w &\in \{0,1\}^{n_c}, \label{eq:eqlf-bin3} 
\end{align}
\label{eq:eqlf-sep-2}
\end{subequations}
where $\Theta_i = \max  \left( \max_j \left| \mysub{DA}{ij} \right| , \max_j \left( \mysub{F}{ij} \right), \max_j \left( \mysub{DC}{ij} \right) \right)$ provides an upper bound on the value of the left-hand-side of~\eqref{eq:eqlf-lamhat2} for any feasible solution to the problem.  The objective~\eqref{eq:eqlf-obj} and constraint~\eqref{eq:eqlf-obj2} minimize the nonconvex quadratic function $(x,\lambda,\hat{\lambda})^T \widehat{M}(P^l) (x,\lambda,\hat{\lambda})$.  Constraints~\eqref{eq:eqlf-split}---\eqref{eq:eqlf-bin1} ensure that $(x,\lambda,\hat{\lambda} \in \mathbb{B}^{n_x + 2 n_c}_1$. The constraints~\eqref{eq:eqlfk1}---\eqref{eq:eqlf-bin2} enforce that $(x,\lambda) \in \gr\,\sol{\StateCompl x}{\ComplCompl}$, and the constraints~\eqref{eq:eqlf-lamhat1}---~\eqref{eq:eqlf-bin3} force that $\hatlambda \in \sol{\StateCompl(A x + \ComplState \lambda)}{\ComplCompl}.$

\
\section{Numerical Experiments}\label{sec:numerics}

We implemented Algorithm~\ref{algo:CPPgen} for the computation of CQLF and EQLF in Python and solved the master problems and separation problems using Gurobi 9.0.2~\cite{gurobi}.  All experiments were run on a Dual-Core Macbook Pro with an i5 processor clocked at 3.1GHz. In our experiments, we chose separation tolerance $\epsilon  = 10^{-6}$ in Algorithm~\ref{algo:CPPgen}. In our implementation, we do not always solve the separation problems for the CQLF and EQLF to optimality.  
We terminate the separation early, if the solver finds a feasible solution whose  objective value is less than or equal to zero.  In this case, the incumbent solution is a vector that can be used to separate the current master problem solution and can be added to the set of vectors ($\setU^l$ or $\setV^l$) in the master problem~\eqref{sicpPgen-mp}.  To initialize the sets $\setU^0$ and $\setV^0$ in our algorithm, we use the $n_x$ unit vectors $e_i$ for initial state variables, and complementarity variables implied by these state variables:
\begin{subequations}
	\begin{align}
	\setU^0 &:= \{ e_1, e_2, \ldots e_n\}\\
	\setV^0_c &:= \{ (e_1, \beta_1), (e_2, \beta_2), \ldots (e_n, \beta_n)\},\\	
	\setV^0_e &:= \{ (e_1, \beta_1, \hat{\beta}_1), (e_2, \beta_2, \hat{\beta}_2), \ldots (e_n, \beta_n, \hat{\beta}_n)\},
	\end{align}
\end{subequations}
where $\beta_i \in \gr\,\sol{\StateCompl e_i}{\ComplCompl}$, and $\hat{\beta}_i \in \sol{\StateCompl(A e_i + \ComplState \beta_i)}{\ComplCompl}$ for all $i=1,\ldots,n_x$.  For CQLF, we initialize with the sets $\setU^0$ and $\setV^0_c$.  For EQLF, the algorithm is initialized with $\setU^0$ and $\setV^0_e$.


\subsection{\dlcs\, derived from LCS with $\tComplCompl$ a P-matrix}
We consider the 3 LCS examples from~\cite{CamlibelPangShen2006} (Examples 3.1-3.3) and another example from Heemels, Schumacher and Weiland~\cite[\S2]{HeemelsSchumacher2000}. 
\begin{itemize}
    \item \texttt{cam31} is Example~3.1~\cite{CamlibelPangShen2006} with data matrices:   
    $\tA = \begin{psmallmatrix} 1.0 \end{psmallmatrix}$, 
    $\tComplState = \begin{psmallmatrix} 2 & -2 \end{psmallmatrix}$,  
    $\tStateCompl = \begin{psmallmatrix} 1 \\ -1 \end{psmallmatrix}$, 
    $\tComplCompl = \begin{psmallmatrix} 1 & 3 \\ 0 & 1 \end{psmallmatrix}$. 
    \item \texttt{cam32} is Example~3.2~\cite{CamlibelPangShen2006} with data matrices:   
    $\tA = \begin{psmallmatrix} -1.0 \end{psmallmatrix}$, 
    $\tComplState = \begin{psmallmatrix} 0 & 1 \end{psmallmatrix}$,  
    $\tStateCompl = \begin{psmallmatrix} 1 \\ 1 \end{psmallmatrix}$, 
    $\tComplCompl = \begin{psmallmatrix} 1 & 3 \\ 0 & 1 \end{psmallmatrix}$. 
    \item \texttt{cam33} is Example~3.3~\cite{CamlibelPangShen2006} with data matrices: 
    $\tA = \begin{psmallmatrix} -5 & -4 & 0 \\ -1 & -2 & 0 \\ 0 & 0 & 1 \end{psmallmatrix}$, 
    $\tComplState = \begin{psmallmatrix} -3 & 0 & 0 \\ -21 & 0 & 0 \\ 0 & 2 & -2 \end{psmallmatrix}$,  
    $\tStateCompl = \begin{psmallmatrix} 1 & 0 & 0 \\ 0 & 0 & 1 \\ 0 & 0 & -1 \end{psmallmatrix}$, 
    $\tComplCompl = \begin{psmallmatrix} 1 & 0 & 0  \\ 0 & 1 & 3 \\ 0 & 0 & 1 \end{psmallmatrix}$.
    \item \texttt{hem2} is Example in \S2~\cite{HeemelsSchumacher2000} with data matrices: 
    $\tA = \begin{psmallmatrix} 0 & 0 & 1 & 0 \\ 0 & 0 & 0 & 1 \\ -2 & 1 & 0 & 0 \\ 1 & -1 & 0 & 0 \end{psmallmatrix}$, 
    $\tComplState = \begin{psmallmatrix} 0 \\ 0 \\ 1 \\ 0 \end{psmallmatrix}$, 
    $\tStateCompl = \begin{psmallmatrix} 1 & 0 & 0 & 0 \end{psmallmatrix}$, 
    $\tComplCompl = \begin{psmallmatrix} 1 \end{psmallmatrix}$.
\end{itemize}
Note that in all the examples above the $\tComplCompl$ matrix is a P-matrix. In the following we consider \dlcs\, that are obtained from explicit ($\theta = 0$ in~\eqref{tslcs}) and implicit Euler ($\theta = 1$ in~\eqref{tslcs}) time-stepping schemes.  We considered two possible time-steps of $\tstep = 0.1$ and $0.05$.  For these choices of time-stepping schemes and $\tstep$, it can be verified that $\ComplCompl$ is also a P-matrix. 

Table~\ref{tab:table1} summarizes the results for the CQLF computation while Table~\ref{tab:table2} provides the results for the EQLF.  Table~\ref{tab:table1} shows that the algorithm is successful in identifying a CQLF such that $\mu \geq \epsilon$ for \texttt{cam31} and \texttt{cam32} for both time-steps. The algorithm shows that a CQLF does not exist such that $\mu \geq \epsilon$ for \texttt{cam33} and \texttt{hem2} when using the explicit Euler time-stepping scheme and either time-step choices. When using the implicit Euler scheme \texttt{cam31}, \texttt{cam32} and \texttt{cam33} permit identifying a CQLF for both time-step choices.  As for \texttt{hem2}, a CQLF does not exist for either time-step choices.  
Table~\ref{tab:table2} shows that an EQLF with $\mu \geq 10^{-6}$ can be computed for all the examples except for \texttt{hem2} for either time-steps when using the explicit Euler. When using the implicit Euler an EQLF is computed for all systems. 

In all cases that are stabilized using CQLF (EQLF), we also simulated the system from 100 random initial and verified that the CQLF (EQLF) is indeed decreasing over the trajectories as shown in Theorems~\ref{thm:CQLF} and~\ref{thm:EQLF}.

From the tables it can be seen that the best $\mu$ for different time-steps (whenever feasible) is in proportion to the time-steps. The next section shows that this is indeed true for the CQLF (see Lemma~\ref{lemma:relateM2tM}).  

\begin{table}[h]
    \centering
    \begin{tabular}{|c|c|c|c|c|c|c|c|}
            \hline
            & & \multicolumn{3}{c|}{$\tstep = 0.1$} & \multicolumn{3}{c|}{$\tstep = 0.05$} \\
            \cline{3-8}
        & Name & Best $\mu$ & \# iters. & Time(s) & Best $\mu$ & \# iters. & Time(s) \\
        \hline
        \multirow{4}{*}{\STAB{\rotatebox[origin=c]{90}{Explicit}}}
	& \texttt{cam31}  & $4.348 \cdot 10^{-2}$ & 1 & 0.006 & $2.326 \cdot 10^{-2}$ & 1 & 0.006 \\
	& \texttt{cam32} & $1.0 \cdot 10^{-1}$ & 1 & 0.003 & $5.0 \cdot 10^{-2}$ & 1 & 0.003 \\
	& \texttt{cam33} & $< 10^{-6}$ & 3 & 0.171 & $< 10^{-6}$ & 10 & 0.477 \\
	& \texttt{hem2} & $< 10^{-6}$& 7 & 0.157 & $< 10^{-6}$ & 6 & 0.118 \\
        \hline
        \multirow{4}{*}{\STAB{\rotatebox[origin=c]{90}{Implicit}}}
	& \texttt{cam31}  & $4.762 \cdot 10^{-2}$ & 1 & 0.003 & $2.439 \cdot 10^{-2}$ & 1 & 0.010 \\
	& \texttt{cam32} & $9.091 \cdot 10^{-2}$ & 1 & 0.004 & $4.762 \cdot 10^{-2}$ & 1 & 0.003 \\
	& \texttt{cam33} & $3.941 \cdot 10^{-3}$ & 9 & 0.624 & $4.373 \cdot 10^{-4}$ & 15 & 1.106 \\
	& \texttt{hem2} & $< 10^{-6}$& 22 & 0.720 & $< 10^{-6}$ & 16 & 0.320 \\        
        \hline
    \end{tabular}
    \caption{Summary of results from applying Algorithm~\ref{algo:CPPgen} for the CQLF computation with $\epsilon = 10^{-6}$. Best $\mu$ - refers to the best $\mu$ obtained at termination; \# iters. - number of times the while-loop. in Algorithm~\ref{algo:CPPgen} is executed;  Tot. CPU time (s) - total time for algorithm.}
    \label{tab:table1}
\end{table}

\begin{table}[h]
    \centering
    \begin{tabular}{|c|c|c|c|c|c|c|c|}
            \hline
            & & \multicolumn{3}{c|}{$\tstep = 0.1$} & \multicolumn{3}{c|}{$\tstep = 0.05$} \\
            \cline{3-8}
        & Name & Best $\mu$ & \# iters. & Time(s) & Best $\mu$ & \# iters. & Time(s) \\
        \hline
        \multirow{4}{*}{\STAB{\rotatebox[origin=c]{90}{Explicit}}}
       
        & \texttt{cam31} & $4.188 \cdot 10^{-2}$ & 2 & 0.048 & $4.133 \cdot 10^{-2}$ & 1 & 0.025 \\
        & \texttt{cam32} & $1.766 \cdot 10^{-1}$ & 1 & 0.014 & $4.870 \cdot 10^{-2}$ & 1 & 0.012 \\
        & \texttt{cam33} & $4.859 \cdot 10^{-3}$ & 9 & 1.630 & $1.112 \cdot 10^{-2}$ & 13 & 2.500 \\
        & \texttt{hem2} & $< 10^{-6}$ & 8 & 0.365 & $< 10^{-6}$ & 6 & 0.321 \\
        \hline
        \multirow{4}{*}{\STAB{\rotatebox[origin=c]{90}{Implicit}}}
        & \texttt{cam31} & $8.453 \cdot 10^{-2}$ & 1 & 0.033 & $4.334 \cdot 10^{-2}$ & 1 & 0.029 \\
        & \texttt{cam32} & $1.607 \cdot 10^{-1}$ & 1 & 0.014 & $8.453 \cdot 10^{-2}$ & 1 & 0.017 \\
        & \texttt{cam33} & $1.2 \cdot 10^{-2}$ & 4 & 0.990 & $6.723 \cdot 10^{-3}$ & 8 & 1.501 \\
        & \texttt{hem2} & $9.629 \cdot 10^{-5}$ & 35 & 7.658 & $2.831 \cdot 10^{-5}$ & 33 & 23.152 \\
        \hline
    \end{tabular}
    \caption{Summary of results from applying Algorithm~\ref{algo:CPPgen} for the EQLF computation with $\epsilon = 10^{-6}$. Best $\mu$ - refers to the best $\mu$ obtained at termination; \# iters. - number of times the while-loop. in Algorithm~\ref{algo:CPPgen} is executed;  Tot. CPU time (s) - total time for algorithm.}
    \label{tab:table2}
\end{table}

\subsection{\dlcs\, with $\ComplCompl$ a Q-matrix and R$_0$-matrix}

Stability results in the literature assume that $\ComplCompl$ is a $P$-matrix \cite{Aydinoglu2020,CamlibelPangShen2006}, which ensures that the state trajectory of the system is unique.  
The sufficient conditions for stability of \dlcs\ given in Section~\ref{sec:stability} assume the weaker condition that $\ComplCompl$ is a $Q$-matrix and an $R_0$ matrix. Here, we demonstrate that the cutting-plane method can find stability conditions for these more general systems by considering the \dlcs
\begin{equation}\label{QP0matrix}
    A = \begin{psmallmatrix} 0.5 & 0.25 \\ -0.25 & 0.5 \end{psmallmatrix}, \,
    \ComplState = \begin{psmallmatrix} 3 \\ 5 \end{psmallmatrix}, \,
    \StateCompl = \begin{psmallmatrix} 1.0 & 0.0 \\ 0.0 & 0.0 \end{psmallmatrix}, \,
    \ComplCompl = \begin{psmallmatrix} 1 & -1 \\ 1 & 0 \end{psmallmatrix}.
\end{equation}
Every principal minor of the matrix $\ComplCompl$ is nonnegative, and hence $\ComplCompl$ is a P$_0$-matrix. Further, $\ComplCompl$ is a R$_0$-matrix since $\sol{0}{\ComplCompl} = \{0\}$.  Thus, $\ComplCompl$ is also a Q-matrix.  Running our cutting-plane algorithm proves that the system cannot be stabilized using a CQLF.
However, the system in~\eqref{QP0matrix} can be stabilized using the EQLF with $\mu = 4.057 \cdot 10^{-3}$ and 
\[
P = \begin{psmallmatrix}
5.33238468 \cdot 10^{-1} & -1.79873472 \cdot 10^{-1} & -1.0 & -4.13688777 \cdot 10^{-2} \\
-1.79873472 \cdot 10^{-1} & 1.70818637 \cdot 10^{-1} &  2.78264033 \cdot 10^{-1} & -8.65180851 \cdot 10^{-4} \\
-1.0  & 2.78264033 \cdot 10^{-1}  & 1.0 & -1.0 \\
 -4.13688777 \cdot 10^{-2} & -8.65180851 \cdot 10^{-4} & -1.0 & 7.53616763 \cdot 10^{-2}
 \end{psmallmatrix}.
\]

\section{LCS and \dlcs\, Stability}\label{sec:lcsdlcsstability}

{\c C}amlibel, Pang and Shen~\cite[Theorem~3.1]{CamlibelPangShen2006} present sufficient conditions for stability using an EQLF for the LCS~\eqref{lcs}.  In this section, we consider the \dlcs~\eqref{dlcs} obtained from the discretization of the LCS~\eqref{lcs} for some $\theta \in [0,1]$ in~\eqref{tslcs1}.  Recall from \eqref{tslcs1} that by defining the matrix $\tAthetah = (I_{n_x} - \theta \tstep \tA)$, the time-stepping formulation of the LCS~\eqref{tslcs1} can be identified with the \dlcs\, as
\begin{equation}
\begin{aligned}
	& A =  I_{n_x} + \tstep \tAthetah^{-1} \tA, \, &&
	\ComplState = \tstep \tAthetah^{-1} \tComplState, \, \\
	& \StateCompl = \tStateCompl (I_{n_x} +\tstep \tAthetah^{-1} \tA),\, &&
	\ComplCompl =\tComplCompl + \tstep \tStateCompl \tAthetah^{-1} \tComplState.
\end{aligned}\label{tslcs2dlcs}
\end{equation}
The stability results in~\cite[Theorem~3.1]{CamlibelPangShen2006} are obtained under the assumption that $\tComplCompl$ is a P-matrix.  Accordingly, we assume in the rest of the section that $\tComplCompl$ is a P-matrix. 
The main result presented in this section is to show that the LCS is 
exponentially stable using a CQLF (EQLF) if and only if the \dlcs\, is exponentially stable using a CQLF (EQLF) for all time steps $\tstep$ sufficiently small and any $\theta \in [0,1]$.

As such, in this section, we assume that the time step $\tstep$ in the solution scheme is chosen such that 
\begin{equation}
\tstep \theta \cdot \|\tA\| < 1,\, \tstep \cdot \|\tAthetah^{-1} \tA\| < 1, \text{ and } \ComplCompl \text { is both a Q-matrix and an R$_0$-matrix}. \label{bndOnDeltat}
\end{equation}
Note that the first inequality in~\eqref{bndOnDeltat} ensures that $\tAthetah^{-1}$ exists while the second inequality ensures that $A$ is invertible.  Both of these conditions can be ensured for all $\tstep$ sufficiently small.  Finally, since we assume that $\tComplCompl$ is a P-matrix, then there must exist $\tstep$ sufficiently small so that $\ComplCompl =  \tComplCompl + \tstep \tStateCompl \tAthetah^{-1} \tComplState$ is also a P-matrix, so it is also both a Q-matrix and an R$_0$-matrix.

To show equivalence of stability for the LCS and DLCS, we require the following lemma relating the inverse of matrices involved in their dynamics. 

\begin{lemma}\label{lemma:boundmatrix}
Let $\tA \in \R^{n_x \times n_x}$, $\tAthetah := I - \theta \tstep \cdot \tA$, and $A := I + \tstep \cdot \tAthetah^{-1} \tA$ be such that $\theta \tstep \cdot \|\tA\| < 1$ and $\tstep \cdot \|\tAthetah^{-1} \tA\| < 1$.  Then 
\begin{equation}
\label{eq:boundnormmatrixinverse}
A^{-1} = I - h \cdot \tA + h^2 \cdot R_{\tA}
\end{equation}
for some matrix $R_{\tA}$ with $\|R_{\tA}\| < \infty$.
\end{lemma}
\begin{proof}
    The lemma follows by twice applying the Neumann power series expansion of the matrix inverse $(I-T)^{-1} = \sum_{k=0}^\infty T^k$.  For ease of notation, let $B := \tAthetah^{-1} \tA$.  By assumption, $\tstep \cdot \|B\| < 1$ so the power series expansion implies that
    \begin{equation}
    \label{eq:bml1}
    A^{-1} = (I+\tstep \cdot B)^{-1} = I - \tstep \cdot B + \tstep^2 \cdot Q_1
    \end{equation}
    for some matrix $Q_1$, with $\|Q_1\| < \infty$.  Since $\theta \tstep \cdot \|\tA\| < 1$, we can apply the power series expansion to $\tAthetah^{-1}$ yielding
    \begin{equation}
    \label{eq:bml2}
    B = \tAthetah^{-1} \tA = (I - \theta \tstep \tA)^{-1} \tA = (I + \tstep \theta \cdot \tA + (\tstep \theta)^2 Q_2) \tA
    \end{equation}
    for some matrix $Q_2$ with $\|Q_2\| < \infty$.
    The two applications of the Neumann series together imply that  
    \begin{align*}
    A^{-1} &= I - hB + h^2 Q_1\\
    &= I - h[\tA + \tstep \theta \cdot \tA^2 + (\tstep \theta)^2 \cdot Q_2 \tA] + h^2 \cdot Q_1\\
    &= I - \tstep \cdot \tA + \tstep^2 [Q_1 - \theta \cdot \tA^2 - \tstep \theta^2 \cdot Q_2 \tA].
    \end{align*}
    The matrix $R_{\tA} := Q_1 - \theta \cdot \tA^2 - \tstep \theta^2 \cdot Q_2 \tA$ has bounded norm, which completes the proof.
\end{proof}

\subsection{CQLF Equivalence}\label{sec:cqlfequivalence}

We will begin by stating the conditions for exponential stability of LCS using aCQLF.  Note that~\cite[Theorem~3.1]{CamlibelPangShen2006} states the asymptotic stability of LCS using an EQLF.  Just as in the \dlcs\, (see remarks following Theorem~\ref{thm:EQLF}) we can obtain the stability results for the CQLF after setting certain matrices in the Lyapunov function to zero.

\begin{proposition}\cite[Theorem~3.1]{CamlibelPangShen2006}
The LCS~\eqref{lcs} is asymptotically stable using the CQLF if there exists $\tP_{xx} \in \Sy^{n_x}$ 
such that 
\begin{subequations}
\begin{align}
& \tP_{xx} \succ 0, \mbox{ and } \label{lcscqlf:setK1} \\
& Q(\tP_{xx}) \coloneqq - \begin{pmatrix} 
\tA^T\tP_{xx} + \tP_{xx} \tA & \tP_{xx} \tComplState \\	
\tComplState^T \tP_{xx} & 0 
\end{pmatrix}\succ_{\setC} 0, \label{lcscqlf:setK2} 
\end{align}
where the set ${\setC}  \subseteq \R^{n_x+n_c}$ is defined as
\begin{equation}
\setC = \gr\,\sol{\tStateCompl x}{\tComplCompl}. \label{deftildesetK}
\end{equation}
\end{subequations}
\end{proposition}


We will begin the analysis by relating the cones $\setK$ in \eqref{defK} and $\setC$ in \eqref{deftildesetK}. 

\begin{lemma}\label{lemma:relatecones}
Suppose $\tComplCompl$ is a P-matrix, and let $\tstep$ be such that~\eqref{bndOnDeltat} holds. 
Then the following hold:
\begin{enumerate}[(a)]
    \item If $(x,\lambda) \in \setK$ then $(\tx,\lambda) \in \setC$, where $\tx =  A x + \ComplState \lambda$.
    \item If $(\tx,\lambda) \in {\setC}$  
    then $(x,\lambda) \in \setK$ where $x = A^{-1}\tx - A^{-1} \ComplState \lambda$.
    \item Suppose $(x,\lambda) \in \setK$ and $\tx = A x + \ComplState \lambda$. There exist constants $c_1, c_2 > 0$ such that 
		$c_1 \|( \tx,\lambda )\|  \leq 
		\|( x , \lambda )\| \leq 
		c_2  \|( \tx , \lambda )\|$. 
\end{enumerate}
\end{lemma}
\begin{proof}
Let the matrix $N$ be defined as
\begin{align}
{N} =   \begin{pmatrix}
            A & C \\
            0 & I_{n_c} 
        \end{pmatrix}. \label{defN}
\end{align}
Let $(x,\lambda) \in \R^{n_x+n_c}$ and define $\tx = Ax + \ComplState \lambda$.  From the definition of~\eqref{defN}, we have that 
\begin{align}
\begin{pmatrix} \tx \\ \lambda \end{pmatrix} = N \begin{pmatrix} x \\ \lambda\end{pmatrix}. \label{relatecones}
\end{align}   
The matrix $N$ is invertible with 
$N^{-1} = \begin{psmallmatrix} A^{-1} & -A^{-1}\ComplState \\ 0 & I_{n_c} \end{psmallmatrix}$.  
The claims in (a) and (b) follow immediately.  From the upper triangular representation of $N$, we have that the eigenvalues of $N$ are the eigenvalues of $A$ and $I_{n_c}$.  Hence, $\|N\| \leq \max(\|A\|,1)$ and $\|N^{-1}\| \leq \max(\|A^{-1}\|,1)$.  Combining this with~\eqref{relatecones} yields the claim in (c).
\end{proof}
%

We now present an alternative representation of the matrix-valued function $M(P_{xx})$ defined in~\eqref{defMmatrix} in terms of 
$Q(\tP_{xx})$ from~\eqref{lcscqlf:setK2}.
\begin{lemma}\label{lemma:relateM2tM}
Let $P_{xx} \in \Sy^{n_x}$ be given. Then for all $\tstep > 0$ sufficiently small  
\begin{equation}
M(P_{xx}) = N^T \left( \tstep \cdot Q(P_{xx}) + \tstep^2 \cdot Q_r(P_{xx}) \right) N, \label{relateM2tM}
\end{equation}
where $Q_r(P_{xx})$ is a matrix of bounded norm, and $N = \begin{psmallmatrix} A & C \\ 0 & I_{n_c} \end{psmallmatrix}$.
\end{lemma}
\begin{proof}
Consider the definition of the matrix-valued function $M(P_{xx})$ in~\eqref{defMmatrix}
\begin{equation}\begin{aligned}
M(P_{xx}) = &\; 
N^{T} N^{-T} 
\begin{pmatrix} A^TP_{xx}A - P_{xx} & A^TP_{xx} \ComplState \\  
\ComplState^T P_{xx} A & \ComplState^T P_{xx} \ComplState \end{pmatrix} N^{-1} N \\
= &\; N^T  \begin{pmatrix} P_{xx} - A^{-T} P_{xx}A^{-1} & A^{-T}P_{xx}A^{-1} \ComplState \\  
\ComplState^T A^{-T} P_{xx} A^{-1} & - \ComplState^T A^{-T} P_{xx} A^{-1} \ComplState \end{pmatrix}  N, 
\end{aligned} \label{coposM}
\end{equation}
where the final equality is obtained by substituting for $N^{-1}$, multiplying through, and simplifying. 
Substituting~\eqref{tslcs2dlcs} into the blocks of the matrix in~\eqref{coposM}, simplifying, and applying Lemma~\ref{lemma:boundmatrix}, we obtain for all $\tstep$ satisfying~\eqref{bndOnDeltat} that

\begin{subequations}
\begin{align}
P_{xx}- A^{-T}P_{xx} A^{-1}
	=&\; \tstep (\tA^T P_{xx} + P_{xx} \tA) + \tstep^2 \cdot Q_{r,11}(P_{xx}) \\
A^{-T}P_{xx}A^{-1}\ComplState
	=&\; \tstep P_{xx} \tComplState + \tstep^2 \cdot Q_{r,12}(P_{xx}) \\
-\ComplState^T A^{-T} P_{xx} A^{-1} \ComplState 
	=&\; \tstep^2 \cdot Q_{r,22}(P_{xx})  
\end{align} \label{simplifyForsmallDt}
\end{subequations}
where $Q_{r,11}(P_{xx})$, $Q_{r,12}(P_{xx})$, and $Q_{r,22}(P_{xx})$ are residual matrices with bounded norm.  Define the matrix $Q_r(P_{xx}) = \begin{psmallmatrix} Q_{r,11}(P_{xx}) & Q_{r,12}(P_{xx}) \\ Q_{r,12}(P_{xx})^T & Q_{r,22}(P_{xx}) \end{psmallmatrix}$.  
Plugging~\eqref{simplifyForsmallDt} into the expressions in~\eqref{coposM} yields~\eqref{relateM2tM}, which proves the claim.
\end{proof}

Combining the Lemmas~\eqref{lemma:relatecones} and~\eqref{lemma:relateM2tM} yields the main result that there exists a CQLF for the LCS if and only if there exists a CQLF for \dlcs.

\begin{theorem}\label{thm:cqlfequiv}
Suppose $\tComplCompl$ is a P-matrix, and $\tstep$ is such that~\eqref{bndOnDeltat} holds. Then the LCS is exponentially stable using a CQLF if and only if the \dlcs\, obtained by the time-stepping formulation~\eqref{tslcs1} is exponentially stable using a CQLF.
\end{theorem}
\begin{proof}
Let $P_{xx} \in \Sy^{n_x}$ be given.  Pre-multiply and post-multiply the equation in~\eqref{relateM2tM} by $(x,\lambda) \in \setK$ to obtain
\begin{equation}
	\begin{pmatrix} x \\ \lambda \end{pmatrix}^T M(P_{xx}) \begin{pmatrix} x \\ \lambda \end{pmatrix} = 
	\begin{pmatrix} \tx \\ \lambda \end{pmatrix}^T \left( \tstep \cdot Q (P_{xx}) + \tstep^2 \cdot 
	Q_r (P_{xx}) \right) 
	\begin{pmatrix} \tx \\ \lambda \end{pmatrix},
	\label{relateM2tM1}
\end{equation}
where $\tx = Ax + \ComplState \lambda$.  Further, by Lemma~\ref{lemma:relatecones}, we have that $(\tx,\lambda) \in \setC$.

Consider the only if part of the claim. 
Suppose the LCS is exponentially stable using a CQLF. Since $-Q(P_{xx}) \succ_{\cal C} 0$ we have that
$-( \tx, \lambda )^T Q (P_{xx}) ( \tx,  \lambda ) 
\geq \gamma \|( \tx , \lambda )\|^2$ for some $\gamma > 0$.  
Choose $\tstep$ smaller than dictated by~\eqref{bndOnDeltat} if necessary so that $\tstep \cdot | ( \tx , \lambda )^T Q_r (P_{xx}) ( \tx , \lambda ) | \leq \gamma/2$.  Then we have for such $\tstep$ from~\eqref{relateM2tM1} that 
$-( x, \lambda )^T M (P_{xx}) ( x, \lambda ) \geq (\gamma/2) \tstep \|( \tx , \lambda )\|^2$.  Combining this with Lemma~\ref{lemma:relatecones}(c) proves the claim.   
The if part of the claim can be shown using similar arguments.  This completes the proof.
\end{proof}

\subsection{EQLF Equivalence}\label{sec:eqlfequivalence}

We  now consider the the EQLF and derive similar results between the LCS and \dlcs.  The condition for exponential stability derived in~\cite{CamlibelPangShen2006} is stated next.

\begin{proposition}\cite[Theorem~3.1]{CamlibelPangShen2006}
The LCS~\eqref{lcs} is exponentially stable if there exists $\tP_{xx} \in \Sy^{n_x}$, $\tP_{x\lambda} \in \R^{n_x \times n_c}$, $\tP_{\lambda\lambda} \in \Sy^{n_c}$ such that 
\begin{subequations}\label{lcseqlf}
\begin{align}
& \underbrace{\begin{pmatrix} \tP_{xx} & \tP_{x\lambda} \\
\tP_{x \lambda}^T & \tP_{\lambda \lambda} \end{pmatrix}}_{=: \tP} \succ_{\setK} 0 \label{lcsstab:setK1e} \\
& - \underbrace{\begin{pmatrix} 
\tA^T\tP_{xx} + \tP_{xx} \tA & \tP_{xx} \tComplState + \tA^T \tP_{x\lambda} & \tP_{x\lambda} \\	
\tComplState^T \tP_{xx} + \tP_{x\lambda}^T \tA & \tP_{x\lambda}^T \tComplState + \tComplState^T \tP_{x\lambda} & 
\tP_{\lambda\lambda} \\
\tP_{x\lambda}^T & \tP_{\lambda\lambda} & 0
\end{pmatrix}}_{=: \hatQ(\tilde{P})} \succ_{\setC} 0
\end{align}
where the set $\widehat{\setC} \subseteq \R^{n_x+2n_c}$ is defined as
\begin{align}
\widehat{\setC} = \left\{ (x,\lambda,\dirlambda) \in \R^{n_x+2n_c} \, \left| \, \begin{aligned} 
& (x,\lambda) \in \gr\,\sol{\tStateCompl x}{\tComplCompl} \\
& \mysub{\tStateCompl (\tA x + \tComplState \lambda) + \tComplCompl \dirlambda}{\alpha(\lambda)} = 0 \\
& 0 \leq \mysub{\dirlambda}{\beta(\lambda)} \perp \mysub{\tStateCompl (\tA x + \tComplState \lambda) + \tComplCompl \dirlambda}{\beta(\lambda)} \geq 0 \\
& \mysub{\dirlambda}{\gamma(\lambda)} = 0 
\end{aligned} \right. \right\} \label{defKtilde}
\end{align}
\end{subequations}
with $\alpha(\lambda), \beta(\lambda), \gamma(\lambda)$ defined according to~\eqref{lcpind} for $\sol{\tStateCompl \tx}{\tComplCompl}$.
\end{proposition}
Note that for $(x,\lambda,\lambda') \in \widehat{\setC}$, $\dirlambda$ is the directional derivative of $\lambda(x)$ along the direction 
$(\tA x + \tComplState \lambda)$ which is the time-derivative of the state~\eqref{lcs:dyn} \cite{CamlibelPangShen2006}.

Suppose $(x,\lambda,\hatlambda) \in \widehat{\setK}$ with $\hatlambda \in \sol{\StateCompl (Ax + \ComplState \lambda)}{\ComplCompl}$, so that $\hatlambda$ is the one-step evolution of the \dlcs~\eqref{dlcs:compl} from the state $(Ax + \ComplState \lambda)$. 
By Lemma~\ref{lemma:relatecones}, we also have that $\lambda = \sol{\tStateCompl \tx}{\tComplCompl}$ where $\tx = Ax + \ComplState \lambda$, i.e. 
$(\tx,\lambda) \in \gr\,\sol{\tStateCompl \tx}{\tComplCompl}$.  Let $\xtraj(t),\lamtraj(t)$ for $t \geq 0$ denote the continuous-time trajectory of the LCS~\eqref{lcs} from the initial condition of $(\xtraj(0),\lamtraj(0)) = (\tx,\lambda)$. Further, let $\dirlambda$ be such that $(\tx,\lambda,\dirlambda) \in \widehat{\setC}$. Then, we intuitively expect that $\hatlambda \approx \lamtraj(h)$ for all sufficiently small $h > 0$. 
We in fact show in Lemma~\ref{lemma:relatehatlambda} that $\hatlambda \approx \lamtraj(0) + h \dirlambda$. Prior to that, we collect preliminary results in Lemma~\ref{lemma:lambdaexpansion} which shows that $\lamtraj(h) \approx \lambda + \tstep \cdot \dirlambda$ and~\ref{lemma:errorbound} which shows that $\lamtraj(h) \approx \hatlambda$.  This results rely on results from 
Shen and Pang~\cite{ShenPang2006} and Chen and Wang~\cite{ChenWang2012} respectively. 

\begin{lemma}(\cite[Lemma~14]{ShenPang2006})\label{lemma:lambdaexpansion}
Suppose $\tComplCompl$ is a P-matrix.  Let $(\xtraj(t),\lamtraj(t))$ for $t \geq 0$ denote the continuous-time trajectory of the LCS~\eqref{lcs} from the initial condition $(\xtraj(0),\lamtraj(0)) = (\tx,\lambda) \in \gr\,\sol{\tStateCompl \tx}{\tComplCompl}$. Then
\begin{equation}
    \|\lamtraj(h) - \lambda - \tstep \cdot \dirlambda\| =  o(\tstep) \cdot O(\|\tx\|)
    \label{lambdaexpansion}
\end{equation}
where $\dirlambda$ is such that $(\tx,\lambda,\dirlambda) \in \widehat{\setC}$.
\end{lemma}
\begin{proof}
The claim follows from Lemma~14 in~\cite{ShenPang2006}.
\end{proof}

\begin{lemma}(\cite[Theorem~3.1]{ChenWang2012})\label{lemma:errorbound}
Suppose $\tComplCompl$ is a P-matrix.  Let $(\xtraj(t),\lamtraj(t))$ for $t \geq 0$ denote the trajectory of the LCS~\eqref{lcs} from the initial condition $(\xtraj(0),\lamtraj(0)) = (\tx,\lambda) \in \gr\,\sol{\tStateCompl \tx}{\tComplCompl}$. Let $\hatlambda$ denote the one time-step evolution of~\eqref{dlcs:compl} from the same initial state $\tx$ i.e. $\hatlambda \in \gr\,\sol{\StateCompl \tx}{\ComplCompl}$.  Then there exists a $\overline{\tstep} > 0$ such that for all $\tstep < \overline{\tstep}$
\begin{equation}
    \|\hatlambda - \lamtraj(\tstep)\| = \tstep^2 \cdot O(\|\tx\|). 
    \label{errorbound}
\end{equation}
\end{lemma}
\begin{proof}
Under the assumptions of the lemma it is readily verified that Theorem~3.1 in~\cite{ChenWang2012} holds.  Hence, $\|\hatlambda - \lambda(\tstep)\|_{\infty} \leq \tstep^2 L \|\tA \tx + \tComplState \lambda\|_\infty$ for some constant $L > 0$ independent of $\tstep$, $x, \lambda$. Using Lemma~\ref{lemma:Pmatprops}, we know that $\|\lambda\|$ can be bounded using $\|\tx\|$. Combining this with the above and using the equivalence of norms yields~\eqref{errorbound}.  
\end{proof}

\noindent Combining Lemma~\ref{lemma:lambdaexpansion} and~\ref{lemma:errorbound} gives us the desired relation between $\hatlambda$ and the pair $\lambda$, $\lambda'$.

\begin{lemma}\label{lemma:relatehatlambda}
Suppose the assumptions of Lemmas~\ref{lemma:lambdaexpansion} and~\ref{lemma:errorbound} hold. Then
\begin{equation}
    \hatlambda = \lambda + \tstep \cdot \dirlambda + o(\tstep) \cdot f(\tx) \label{relatehatlambda}
\end{equation}
where $\|f(\tx)\| = O(\|\tx\|)$.
\end{lemma}
\begin{proof}
The claim follows from combining Lemma~\ref{lemma:lambdaexpansion} and~\ref{lemma:errorbound}. 
\end{proof}

We next relate the set $\widehat{\setK}$ in~\eqref{defKhat} with $\widehat{\setC}$.
\begin{lemma}\label{lemma:relateKhat2Kbar}
Suppose $\tComplCompl$ is a P-matrix, and let $\tstep$ be such that~\eqref{bndOnDeltat} holds.  Let the matrix $\widehat{N}$ and vector $\widehat{f}(\tx)$ be defined as
\begin{subequations}
\begin{align}
\widehat{N} =         \begin{pmatrix}
            A & C & 0 \\
            0 & I_{n_c} & 0 \\
            0 & -\frac{1}{\tstep}I_{n_c} & \frac{1}{\tstep}I_{n_c} 
        \end{pmatrix} \text{ and } 
        \widehat{f}(\tx) = \begin{pmatrix} 
        0 \\
        0 \\
        f(\tx)
        \end{pmatrix}. \label{defhatNh}
\end{align}
Then there exists a $\overline{\tstep} > 0$ such that for all $\tstep \in (0,\overline{\tstep})$ the following hold:
\begin{enumerate}[(a)]
    \item If $(\tx,\lambda,\dirlambda) \in \widehat{\setC}$ then 
        $\widehat{N}^{-1}  \begin{psmallmatrix} \tx \\ \lambda \\ \dirlambda \end{psmallmatrix} + 
        o(\tstep) \widehat{f}(\tx) 
        \in \widehat{\setK}$.
    \item If $(x,\lambda,\hatlambda) \in \widehat{\setK}$ then 
        $\widehat{N} \begin{psmallmatrix} x \\ \lambda \\ \hatlambda \end{psmallmatrix} - 
        \frac{o(\tstep)}{\tstep} \widehat{f}(Ax + \ComplState\lambda) 
        \in \widehat{\setC}$.
 	\item There exist constants $c_3, c_4 > 0$ such that 
		$c_3 \|( \tx , \lambda , \dirlambda )\| \leq 
		\|( x , \lambda , \hatlambda )\| \leq 
		c_4  \|( \tx , \lambda , \dirlambda )\|$. 
\end{enumerate}
\end{subequations}
\end{lemma}
\begin{proof}
\begin{subequations}
From Lemma~\ref{lemma:relatecones} we have that $(A^{-1}\tx - A^{-1}\ComplState\lambda,\lambda) \in \setK$.  Combining this with Lemma~\ref{lemma:relatehatlambda} and noting that
 $\widehat{N}^{-1} = \begin{psmallmatrix} A^{-1} & -A^{-1}C & 0 \\
 0 & I_{n_c} & 0 \\
 0 & I_{n_c} & \tstep \cdot I_{n_c} \end{psmallmatrix}$ yields
 \begin{align}
 \begin{pmatrix} x \\ \lambda \\ \hatlambda \end{pmatrix} =&\; 
\widehat{N}^{-1} \begin{pmatrix} \tx \\ \lambda \\ \dirlambda \end{pmatrix} + 
o(\tstep) \widehat{f}(\tx) \label{relatedirlamhatlam.1} \\
 \begin{pmatrix} \tx \\ \lambda \\ \dirlambda \end{pmatrix} =&\; 
\widehat{N} \begin{pmatrix} x \\ \lambda \\ \hatlambda \end{pmatrix} - 
\frac{o(\tstep)}{\tstep} \widehat{f}(A x + \ComplState \lambda). \label{relatedirlamhatlam.2} 
 \end{align}
where the second equality is obtained by left multiplying~\eqref{relatedirlamhatlam.1} with $\widehat{N}$.  
 The claims in (a) and (b) follow from~\eqref{relatedirlamhatlam.1} and~\eqref{relatedirlamhatlam.2} respectively. 
To prove (c), note that the first and second rows and columns of $\widehat{N}$ can be swapped to produce a lower triangular matrix $\begin{psmallmatrix}
            I_{n_c} & 0 & 0 \\
            C & A & 0 \\
            -\frac{1}{\tstep}I_{n_c} & 0 & \frac{1}{\tstep}I_{n_c} 
        \end{psmallmatrix}$.  
        Based on the lower triangular representation, the eigenvalues of $\widehat{N}$ can be obtained as the eigenvalues of $A$, $I_{n_c}$, and $(1/\tstep) I_{n_c}$.  Hence, $\|\widehat{N}^{-1}\| \leq \max(\|A^{-1}\|,1)$. 
        The fact that $\|\widehat{N}^{-1}\|$ is bounded allows us to establish the existence of constant $c_4$ satisfying $\|(x,\lambda,\hatlambda)\| \leq  c_4 \|(\tx,\lambda,\dirlambda)\|$ in (c).  To show the existence of constant $c_3$ in (c), note that from Lemma~\ref{lemma:relatecones}(c) $c_1 \|(\tx,\lambda)\| \leq \|(x,\lambda)\|$. Further, the directional derivative $\dirlambda$ defined in~\eqref{defKtilde} satisfies a mixed LCP and can be bounded in terms of $(\tA \tx + \tComplState \lambda)$ using Lemma~\ref{lemma:Pmatprops}.  This completes the proof.
 \end{subequations}
\end{proof}

\noindent The matrices $\hatQ$ and $\widehat{M}$ are related next.
\begin{lemma}\label{lemma:relateMhat2Mbar}
Suppose $\tstep$ is chosen to be sufficiently small such that~\eqref{bndOnDeltat} holds.  Let $P_{xx} \in \Sy^{n_x}$, $P_{x\lambda} \in \R^{n_x \times n_c}$, and $P_{\lambda\lambda} \in \Sy^{n_c}$ be given.  Then 
\begin{equation}
\widehat{M}(P) = \widehat{N}^T \left( \tstep \cdot \hatQ(P) + \tstep^2 \cdot \hatQ_r(P) \right) \widehat{N} \label{relateMhat2Mbar}
\end{equation}
where $\hatQ_r(P)$ is a matrix of bounded norm, and $\widehat{N}$ is as defined in~\eqref{defhatNh}.
\end{lemma}
\begin{proof}
Consider the definition of the matrix $\widehat{M}(P)$ in~\eqref{defMhatmatrix}
\begin{equation}\begin{aligned}
\widehat{M}(P)
= &\; \widehat{N}^{T} \widehat{N}^{-T} \begin{pmatrix}
		A^TP_{xx} A - P_{xx} & A^TP_{xx}\ComplState - P_{x\lambda} &  A^TP_{x\lambda} \\
		\ComplState^TP_{xx}A - P_{x\lambda}^T & \ComplState^TP_{xx}\ComplState - P_{\lambda\lambda} & \ComplState^TP_{x\lambda} \\
		 P_{x\lambda}^TA& P_{x\lambda}^T\ComplState & P_{\lambda\lambda} 
		 \end{pmatrix} \widehat{N}^{-1} \widehat{N} \\
= &\; \widehat{N}^T  \begin{pmatrix} 
		P_{xx} - A^{-T} P_{xx}A^{-1} & U_{x\lambda} 
		& \tstep P_{x\lambda} \\  
		U_{x\lambda}^T & 
		U_{\lambda\lambda} &
		\tstep P_{\lambda\lambda} \\
		\tstep P_{x\lambda}^T & \tstep P_{\lambda\lambda} & 0 
		\end{pmatrix}  \widehat{N}, \\
\text{where} &\; U_{x\lambda} = A^{-T}P_{xx}A^{-1} \ComplState - A^{-T}P_{x\lambda} + P_{x\lambda} \text{ and }\\
		&\; U_{\lambda\lambda} = - \ComplState^T A^{-T} P_{xx} A^{-1} \ComplState + P_{x\lambda}^TA^{-1} 
		\ComplState + \ComplState^TA^{-T}P_{x\lambda} .
\end{aligned} \label{coposMhat}
\end{equation}
The second equality is obtained from the first by multiplying through and simplifying. 
Substituting~\eqref{tslcs2dlcs} into the blocks of the matrix in~\eqref{coposMhat}, simplifying, and applying the result of Lemma~\ref{lemma:boundmatrix}, we obtain for all $\tstep$ satisfying~\eqref{bndOnDeltat} that 
\begin{subequations}
\begin{align}
P_{xx}- A^{-T}P_{xx} A^{-1}
	=&\; \tstep (\tA^T P_{xx} + P_{xx} \tA) + \tstep^2 \cdot \hatQ_{r,11}, \\
U_{x\lambda}
	=&\; \tstep P_{xx} \tComplState + \tstep^2 \cdot \hatQ_{r,12}, \text{ and} \\
U_{\lambda\lambda}
	=&\; \tstep \left(P_{x\lambda}^T \StateCompl + \StateCompl^T P_{x\lambda} \right) +  
	\tstep^2 \cdot \hatQ_{r,22} , 
\end{align} \label{simplifyForsmallDt1}
\end{subequations}
where $\hatQ_{r,11}$, $\hatQ_{r,12}$, and $\hatQ_{r,22}$ are bounded residual matrices.  Define the matrix $\hatQ_r = \begin{psmallmatrix} \hatQ_{r,11} & \hatQ_{r,12} & 0 \\ \hatQ_{r,12}^T & \hatQ_{r,22} & 0 \\ 0 & 0 & 0 \end{psmallmatrix}$.  
Plugging~\eqref{simplifyForsmallDt1} into the expressions in~\eqref{coposMhat} yields~\eqref{relateMhat2Mbar}, which proves the claim.
\end{proof}

We are now ready to show the equivalence between the EQLF conditions for the LCS~\eqref{lcseqlf} and the \dlcs~\eqref{feasP}.

\begin{theorem}
Suppose $\tComplCompl$ is a P-matrix. Then for all $\tstep$ sufficiently small, the LCS is exponentially stable using a EQLF if and only if the \dlcs\, obtained by the time-stepping formulation~\eqref{tslcs} is exponentially stable using a EQLF.
\end{theorem}
\begin{proof}
Pre-multiply and post-multiply the equation in~\eqref{relateMhat2Mbar} by $(x,\lambda,\hatlambda) \in \widehat{\setK}$ to obtain
\begin{align}
	&\; \begin{pmatrix} x \\ \lambda \\ \hatlambda \end{pmatrix}^T \widehat{M} \begin{pmatrix} x \\ \lambda \\ \hatlambda 
	\end{pmatrix} \nonumber \\
	= &\; 
	\left( \begin{pmatrix} \tx \\ \lambda \\ \dirlambda \end{pmatrix} + \frac{o(\tstep)}{\tstep} \widehat{f}(\tx) \right)^T 
	\left( \tstep \cdot \hatQ + \tstep^2 \cdot \hatQ_r \right) 
	\left( \begin{pmatrix} \tx \\ \lambda \\ \dirlambda \end{pmatrix} + \frac{o(\tstep)}{\tstep} \widehat{f}(\tx) \right) 
	\nonumber \\
	= &\; \tstep \begin{pmatrix} \tx \\ \lambda \\ \dirlambda \end{pmatrix}^T \hatQ 
	\begin{pmatrix} \tx \\ \lambda \\ \dirlambda \end{pmatrix} + o(\tstep) 
	\begin{pmatrix} \tx \\ \lambda \\ \dirlambda \end{pmatrix}^T \hatQ \widehat{f}(\tx) + \frac{o(\tstep)^2}{\tstep}
	\widehat{f}(\tx)^T \hatQ \widehat{f}(\tx) \nonumber \\
	&\; + \tstep^2 	
	\left( \begin{pmatrix} \tx \\ \lambda \\ \dirlambda \end{pmatrix} + \frac{o(\tstep)}{\tstep} \widehat{f}(\tx) \right)^T 
	\hatQ_r 
	\left( \begin{pmatrix} \tx \\ \lambda \\ \dirlambda \end{pmatrix} + \frac{o(\tstep)}{\tstep} \widehat{f}(\tx) \right) 
	\label{relateMhat2Mbar1}
\end{align}
where $\tx = Ax + \ComplState \lambda$.  The first equality follows from Lemma~\ref{lemma:relateMhat2Mbar} for all $\tstep$ sufficiently small.  Further by Lemma~\ref{lemma:relateKhat2Kbar}, we have that $(\tx,\lambda,\dirlambda) \in \widehat{\setC}$.

Consider the only if part of the claim.  Suppose LCS is exponentially stable using an EQLF.  Since $-\widehat{Q} \succ_{\widehat{\cal C}} 0$ there exists $\gamma > 0$ such that $-( \tx,  \lambda, \dirlambda)^T \hatQ (\tx, \lambda, \dirlambda) \geq \gamma \|( \tx, \lambda, \dirlambda )\|^2$.  Then for all $\tstep$ sufficiently small the norm of the sum of the terms other than the leading term on the right hand side of~\eqref{relateMhat2Mbar1} can be made smaller than or equal to $(\gamma \tstep)/2 \|(\tx,\lambda,\dirlambda)\|^2$.  Then for all for such $\tstep$ sufficiently small we have from~\eqref{relateMhat2Mbar1} and Lemma~\ref{lemma:relateKhat2Kbar}(c) that 
\[ -(x, \lambda, \hatlambda )^T \widehat{M} ( x, \lambda, \hatlambda) \geq (\gamma \tstep)/2 \|( \tx, \lambda, \dirlambda )\|^2 \geq (\gamma \tstep)/(2 c_4) \|(x,\lambda,\hatlambda)\|^2. \] 
Thus, $-\widehat{M} \succ_{\widehat{\cal K}} 0$ proving the only if part of the claim.

The proof of the if part of the claim can be obtained in a similar manner.  This completes the proof.
\end{proof}

\section{Conclusions \& Future Work}\label{sec:conclusions}

In this work, we derived sufficient conditions for the Lyapunov stability of a Discrete-Time Linear Complementarity System (DLCS), using both common and extended quadratic Lyapunov functions. We also showed the equivalence between the stability conditions of a Linear Complementarity System and its discrete-time analog for all sufficiently small time-steps.  We proposed a cutting plane algorithm to find a Lyapunov function verifying exponential stability by separating points from nonconvex copositive cones, and the algorithm was demonstrated on small \dlcs\, instances.  
The current investigation opens up a number of avenues for future exploration.
\begin{itemize}
\item First, the cutting plane algorithm can be slow on large instances of the \dlcs. For large instances, the algorithm often requires performing many rounds of cuts before terminating.  Improving the behavior of the cutting plane algorithm is an important future direction of this work. 
\item Another natural area of improvement for this approach is that a Lyapunov function is only found upon termination of the algorithm.  In other words, we do not yet have a heuristic approach for constructing a feasible matrix $P_{xx}$ using information obtained from the cutting plane method.  Such an approach can be advantageous since for stability, we are only interested in a feasible solution satisfying the copositivity conditions.  
\item The incorporation of feedback control and the joint computation of stabilizing controller and Lyapunov function is another avenue for investigation.  The inclusion of feedback controls complicates the separation problem since the matrices in the LCP constraints are now dependent on the feedback matrices.  This necessitates the development of a different algorithm for the computation of Lyapunov function.
\end{itemize}

\appendix

\section{S-Lemma Formulations for CQLF and EQLF}\label{sec:appendix}

\subsection{Common Quadratic Lyapunov Function}

The conditions for exponential stability in Theorem~\ref{thm:CQLF}$(iii)$ can be written as the feasibility problem~\eqref{feasPxx}.  
Moreover, the inequality in~\eqref{feasPxx:setK2} can be relaxed into a linear matrix inequality using the S-Lemma~\cite{BoydLMI}.  Consider the matrix $H \in \R^{(n_x + n_c)}$ defined as
\begin{subequations}
\begin{equation}
	H = \begin{pmatrix} \StateCompl & \ComplCompl \\ 0 & I_{n_c} \end{pmatrix}.
\end{equation}
Consider the Linear Matrix Inequality (LMI) 
\begin{align}
M + H^T WH \prec 0 
\text{ where } W \in \Sy^{(n_x+n_c)}, W \geq 0. \label{feasPxxlmi} 
\end{align}
Let $\conv (\gr\,\sol{\StateCompl x}{\ComplCompl})$ denote the convex hull of $\gr\,\sol{\StateCompl x}{\ComplCompl}$, i.e. 
\[
\conv (\gr\,\sol{\StateCompl x}{\ComplCompl}) = \{ (x,\lambda) \,|\, \lambda \geq 0,\, \StateCompl x + \ComplCompl \lambda \geq 0 \}.
\]

We show that satisfaction of~\eqref{feasPxxlmi} implies that $\psi(x,\lambda) < 0$ for all $(x,\lambda) \in \conv ( \gr\,\sol{\StateCompl x}{\ComplCompl} )$. Thus, we obtain a relaxation of the copositivity conditions in~\eqref{feasPxx} to a LMI~\eqref{feasPxxlmi}.  Multiplying the matrix inequality in~\eqref{feasPxxlmi} from the left and right by $(x,\lambda) \in \conv(\gr\,\sol{\StateCompl x}{\ComplCompl})$ we obtain 
\begin{align}
	&\begin{pmatrix} x \\ \lambda \end{pmatrix}^T \left( M  + H^T WH \right) 
	\begin{pmatrix} x \\ \lambda \end{pmatrix} \nonumber 
	< 0 \\
	\implies& \psi(x,\lambda) + 
	\begin{pmatrix} \StateCompl x + \ComplCompl\lambda \\ \lambda \end{pmatrix}^T 
	W 
	\begin{pmatrix} \StateCompl x + \ComplCompl\lambda \\ \lambda \end{pmatrix} < 0 \nonumber \\	
	\implies& \psi(x,\lambda) < 0 
	\label{cqlf.imp}
\end{align}
where the second implication follows from $(x,\lambda) \in \conv(\gr\,\sol{\StateCompl x}{\ComplCompl})$ and $W_{xx}, W_{\lambda\lambda} \geq 0$.  Hence, satisfaction of~\eqref{feasPxxlmi} $\implies$ \eqref{feasPxx:setK2}.  The computation of a $P_{xx}$ satisfying~\eqref{feasPxx} can be cast as the LMI system
\begin{align}
\text{Find} &\; P_{xx} \in \Sy^{n_x}, W \in \Sy^{n_x+n_c} \\
\text{ such that} &\; 	P_{xx} \succ 0, W \geq 0, \eqref{feasPxxlmi}. 
\end{align}
\end{subequations}

\subsection{Extended Quadratic Lyapunov Function}

\begin{subequations}
Similar to the use of S-Lemma for the CQLF the copositivity conditions in~\eqref{feasP} can be turned into a feasibility problem involving LMIs.  Consider the matrices $J, J' \in \R^{(n_x+n_c) \times (n_x+2n_c)}$, defined as
\begin{equation}
	J = \begin{pmatrix} \StateCompl & \ComplCompl & 0 \\ 0 & I_{n_c} & 0 \end{pmatrix},\,
	J' = \begin{pmatrix} \StateCompl A & \StateCompl \ComplState & \ComplCompl \\ 0 & 0 & I_{n_c} \end{pmatrix}.
\end{equation}
It can be verified that 
\begin{equation}
\left. \begin{aligned}
W_1  \in \Sy^{n_x+n_c},  W_1 \geq 0 \\
P  - H^T {W}_1H \succ 0 
\end{aligned} \right\} \implies \eqref{feasP:setK1} 
\label{eqlf.imp1}
\end{equation}
\begin{equation}
\left. \begin{aligned}
W_2, {W}_3 \in \Sy^{(n_x+n_c)}, W_2, W_3 \geq 0 \\
\widehat{M}  + J^T W_2J + (J')^TW_3(J') \prec 0 
\end{aligned} \right\} \implies \eqref{feasP:setK2}.
\label{eqlf.imp2}
\end{equation}
\end{subequations}
Note that~\eqref{eqlf.imp1} can be derived in a manner similar to that in~\eqref{cqlf.imp}.  
We derive~\eqref{eqlf.imp2} in the following. Multiply the matrices in the left hand side of~\eqref{eqlf.imp2} by $(x,\lambda,\hatlambda) \in \widehat{\setK}$ and rearrange to obtain
 \[\begin{aligned}
     & \begin{pmatrix} x \\ \lambda \\ \hatlambda \end{pmatrix}^T 
     \left( \widehat{M}  + J^T W_2J + (J')^TW_3(J') \right)
     \begin{pmatrix} x \\ \lambda \\ \hatlambda \end{pmatrix} < 0 \\
\implies & \widehat{\psi}(x,\lambda,\hatlambda) + 
    \begin{pmatrix} \StateCompl x + \ComplCompl \lambda \\ \lambda \end{pmatrix}^T W_2 
    \begin{pmatrix} \StateCompl x + \ComplCompl \lambda \\ \lambda \end{pmatrix}    + \\
    & 
    \begin{pmatrix} \StateCompl A x + \StateCompl \ComplState \lambda + \ComplCompl \hatlambda \\ \hatlambda \end{pmatrix}^T W_3 
    \begin{pmatrix} \StateCompl A x + \StateCompl \ComplState \lambda + \ComplCompl \hatlambda \\ \hatlambda \end{pmatrix} < 0 \\
\implies & \widehat{\psi}(x,\lambda,\hatlambda) < 0    
 \end{aligned}
 \]
where the second implication follows from $\lambda \in \sol{\StateCompl x}{\ComplCompl})$, $\hatlambda \in \sol{\StateCompl Ax + \StateCompl \ComplState \lambda}{\ComplCompl}$ and $W_{2}, W_3 \geq 0$.  This proves the implication in~\eqref{eqlf.imp2}.  Just as in the case of the CQLF we can show that the LMI in~\eqref{eqlf.imp2} implies that $\widehat{\psi}(x,\lambda,\hatlambda) < 0$ for all $(x,\lambda,\hatlambda) \in \conv{\widehat{\setK}}$.

\bibliographystyle{plain}
\bibliography{lcs_refs, hybridmpc_refs}

\begin{thebibliography}{10}

\bibitem{AcaryBrogliatoBook2008}
V.~Acary and B.~Brogliato.
\newblock {\em Numerical Methods for Nonsmooth Dynamical Systems}.
\newblock Lecture Notes in Applied and Computational Mechanics. Springer-Verlag
  Berlin Heidelberg, first edition, 2008.

\bibitem{Aydinoglu2020}
Alp Aydinoglu, Victor~M Preciado, and Michael Posa.
\newblock {Contact-Aware Controller Design for Complementarity Systems}.
\newblock In {\em International Conference on Robotics and Automation (ICRA)},
  2020.

\bibitem{BoydLMI}
S.~Boyd, L.~El~Ghaoui, E.~Feron, and V.~Balakrishnan.
\newblock {\em Linear Matrix Inequalities in System and Control Theory}.
\newblock SIAM, 1994.

\bibitem{Branicky1998}
M.~Branicky.
\newblock Multiple lyapunov functions and other analysis tools for switched and
  hybrid systems.
\newblock {\em IEEE Trans. Automat. Control}, 43:475–482, 1998.

\bibitem{Brogliato2003}
B.~Brogliato.
\newblock Some perspectives on analysis and control of complementarity systems.
\newblock {\em IEEE Trans. Automat. Control}, 48:918–935, 2003.

\bibitem{BundfussDur2009}
Stefan Bundfuss and Mirjam D{\"{u}}r.
\newblock {An Adaptive Linear Approximation Algorithm for Copositive programs}.
\newblock {\em SIAM J Optimization}, 20(1):30--53, 2009.

\bibitem{BundfussDurLyapunov2009}
Stefan Bundfuss and Mirjam D{\"{u}}r.
\newblock {Copositive Lyapunov functions for switched systems over cones}.
\newblock {\em Systems and Control Letters}, 58(5):342--345, 2009.

\bibitem{CamlibelPangShen2006}
M.~K. {{\c C}amlibel}, J-S. Pang, and J.~Shen.
\newblock {Lyapunov stability of complementarity and extended systems}.
\newblock {\em SIAM Journal on Optimization}, 17(4):1056--1101, 2006.

\bibitem{CamlibelThesis2001}
M.~K. {\c C}amlibel.
\newblock {\em Complementarity Methods in the Analysis of Piecewise Linear
  Dynamical Systems}.
\newblock PhD thesis, Center for Economic Research, Tilburg University, 2001.

\bibitem{CamlibelHeemels2003}
M.~K. {\c C}amlibel, W.~P. M.~H. Heemels, A.~J. van~der Schaft, and J.~M.
  Schumacher.
\newblock Switched networks and complementarity.
\newblock {\em IEEE Trans. Circuits Systems I: Fund. Theory Appl.},
  50:1036–1046, 2003.

\bibitem{CamlibelSchumacher2004}
M.~K. {\c C}amlibel and J.~M. Schumacher.
\newblock {\em Open Problems in Mathematical Systems and Control Theory},
  chapter Copositive Lyapunov functions, page 189–193.
\newblock Princeton University Press, Princeton, NJ, 2004.

\bibitem{KennethPilwon2017}
Kenneth Chao and Pilwon Hur.
\newblock {A step towards generating human-like walking gait via trajectory
  optimization through contact for a bipedal robot with one-sided springs on
  toes}.
\newblock In {\em IEEE International Conference on Intelligent Robots and
  Systems}, pages 4848--4853, 2017.

\bibitem{ChenWang2012}
Xiaojun Chen and Zhengyu Wang.
\newblock {Computational error bounds for a differential linear variational
  inequality}.
\newblock {\em IMA Journal of Numerical Analysis}, 32(3):957--982, 2012.

\bibitem{CottlePangStone}
R.W. Cottle, J-S. Pang, and R.E. Stone.
\newblock {\em The Linear Complementarity Problem}.
\newblock SIAM, 1992.

\bibitem{DecarloBranicky2000}
R.~A. Decarlo, M.~S. Branicky, S.~Pettersson, and B.~Lennartson.
\newblock Perspectives and results on the stability and stabilization of hybrid
  systems.
\newblock {\em Proceedings of the IEEE}, 88:1069–1082, 2000.

\bibitem{EichfelderPovh2013}
Gabriele Eichfelder and Janez Povh.
\newblock {On the set-semidefinite representation of nonconvex quadratic
  programs over arbitrary feasible sets}.
\newblock {\em Optimization Letters}, 7(6):1373--1386, 2013.

\bibitem{gurobi}
LLC Gurobi~Optimization.
\newblock Gurobi optimizer reference manual, 2020.

\bibitem{HanTiwariCamlibel2009}
Lanshan Han, Alok Tiwari, M.~Kanat Camlibel, and Jong-Shi Pang.
\newblock Convergence of time-stepping schemes for passive and extended linear
  complementarity systems.
\newblock {\em SIAM Journal on Numerical Analysis}, 47(5):3768--3796, 2009.

\bibitem{HeemelsThesis2000}
W.~Heemels.
\newblock {\em {Linear complementarity systems: A Study in Hybrid Dynamics}}.
\newblock Phd thesis, Technische Universiteit Eindhoven, 2000.

\bibitem{HeemelsSchumacher2000}
W.~P.M.H. Heemels, J.~M. Schumacher, and S.~Weiland.
\newblock {Linear complementarity systems}.
\newblock {\em SIAM Journal on Applied Mathematics}, 60(4):1234--1269, 2000.

\bibitem{Hespanha2004}
J.~P. Hespanha.
\newblock Uniform stability of switched linear systems: Extension of
  lasalle’s invariance principle.
\newblock {\em IEEE Trans. Automat. Control}, 49:470–482, 2004.

\bibitem{Hespanha2005}
J.~P. Hespanha, D.~Liberzon, D.~Angeli, and E.~D. Sontag.
\newblock Nonlinear norm-observability notions and stability of switched
  systems.
\newblock {\em IEEE Trans. Automat. Control}, 50:154–168, 2005.

\bibitem{JohanssonRantzer1998}
M.~Johansson and A.~Rantzer.
\newblock Computation of piecewise quadratic lyapunov functions for hybrid
  systems.
\newblock {\em IEEE Trans. Automat. Control}, 43:555–559, 1998.

\bibitem{khalil2002nonlinear}
H.K. Khalil.
\newblock {\em Nonlinear Systems}.
\newblock Pearson Education. Prentice Hall, 2002.

\bibitem{Liberzon2003}
D.~Liberzon.
\newblock {\em Switching in Systems and Control}.
\newblock Birkhauser Basel, first edition, 2003.

\bibitem{ManchesterKuindersma2017}
Zachary Manchester and Scott Kuindersma.
\newblock {Variational Contact-Implicit Trajectory Optimization Variational
  Contact-Implicit Trajectory Optimization}.
\newblock {\em International Symposium on Robotics Research (ISRR)},
  (December):1--16, 2017.

\bibitem{Mayne2000}
D.Q. Mayne, J.B. Rawlings, C.V. Rao, and P.O.M. Scokaert.
\newblock {Constrained model predictive control: Stability and optimality}.
\newblock {\em Automatica}, 36(6):789--814, jun 2000.

\bibitem{MordatchTodorov2012}
I.~Mordatch, E.~Todorov, and Z.~Popovi\'{c}.
\newblock Discovery of complex behaviors through contact-invariant
  optimization.
\newblock {\em ACM Transactions on Graphics (TOG)}, 31(4 Article: 43), 2012.

\bibitem{MurtyKabadi1987}
K.G. Murty and S.N. Kabadi.
\newblock Some np-complete problems in quadratic and nonlinear programming.
\newblock {\em Mathematical Programming}, 39:117--129, 1987.

\bibitem{NagurneyZhang1996}
Anna Nagurney and Ding Zhang.
\newblock {\em Projected Dynamical Systems and Variational Inequalities with
  Applications}.
\newblock International Series in Operations Research \& Management Science.
  Springer, first edition, 1996.

\bibitem{PangShen2007}
Jong-Shi Pang and Jinglai Shen.
\newblock Strongly regular differential variational systems.
\newblock {\em IEEE Transactions on Automatic Control}, 52, 2007.

\bibitem{PangStewart2008}
Jong~Shi Pang and David~E. Stewart.
\newblock {Differential variational inequalities}.
\newblock {\em Mathematical Programming}, 113(2):345--424, 2008.

\bibitem{PatelShield2019}
A.~Patel, S.~L. Shield, S.~Kazi, A.~M. Johnson, and LT~Biegler.
\newblock Contact-implicit trajectory optimization using orthogonal
  collocation.
\newblock {\em IEEE Robotics and Automation Letters}, 4(2):2242--2249, 2019.

\bibitem{Posa14}
M.~Posa, C.~Cantu, and R.~Tedrake.
\newblock A direct method for trajectory optimization of rigid bodies through
  contact.
\newblock {\em The International Journal of Robotics Research (IJRR)},
  33(1):69--81, 2014.

\bibitem{posa_thesis}
Michael Posa.
\newblock {\em Optimization for Control and Planning of Multi-contact Dynamic
  Motion}.
\newblock PhD thesis, Massachusetts Institute of Technology, 2017.

\bibitem{RaghunathanDVI2006}
A~U Raghunathan, J.~R. P\'{e}rez-Correa, E.~Agosin, and L.~T. Biegler.
\newblock Parameter estimation in metabolic flux balance models for batch
  fermentation—formulation \& solution using differential variational
  inequalities (dvis).
\newblock {\em Annals of Operations Research}, 148(1):251--270, 2006.

\bibitem{Schumacher2004}
J.~M. Schumacher.
\newblock Complementarity systems in optimization.
\newblock {\em Math. Program. Ser. B}, 101:263–296, 2004.

\bibitem{ShenPang2006}
Jinglai Shen and Jong~Shi Pang.
\newblock {Linear complementarity systems: Zeno states}.
\newblock {\em SIAM Journal on Control and Optimization}, 44(3):1040--1066,
  2006.

\bibitem{Shorten2007}
Robert Shorten, Fabian Wirth, Oliver Mason, Kai Wulff, and Christopher King.
\newblock {Stability criteria for switched and hybrid systems}.
\newblock {\em SIAM Review}, 49(4):545--592, 2007.

\bibitem{VanSchaftSchumacher1998}
A.~J. van~der Schaft and J.~M. Schumacher.
\newblock Complementarity modeling of hybrid systems.
\newblock {\em IEEE Trans. Automat. Control}, 43:483–490, 1998.

\bibitem{VanSchaftSchumacher2000}
A.~J. van~der Schaft and J.~M. Schumacher.
\newblock {\em An Introduction to Hybrid Dynamical Systems}.
\newblock Springer-Verlag, London, 2000.

\bibitem{VascaIannelli2009}
F.~Vasca, L.~Iannelli, M.~K. {\c C}amlibel, and R~Frasca.
\newblock A new perspective for modeling powerelectronics converters:
  Complementarity framework.
\newblock {\em IEEE Transactions on Power Electronics}, 24(1-2):456--468, 2009.

\end{thebibliography}

\end{document}